\documentclass{llncs}
\usepackage{epsfig}
\usepackage{geometry}

\usepackage{placeins}

\usepackage{hyperref}

\usepackage{amsmath}
\usepackage{amsfonts}
\usepackage{amssymb}

\usepackage{algorithm}
\usepackage{algpseudocode}

\usepackage{etoolbox}
\usepackage{tikz}
\usetikzlibrary{tikzmark}
\usetikzlibrary{calc}

\errorcontextlines\maxdimen
\newcommand{\ALGtikzmarkcolor}{black}
\newcommand{\ALGtikzmarkextraindent}{4pt}
\newcommand{\ALGtikzmarkverticaloffsetstart}{-.5ex}
\newcommand{\ALGtikzmarkverticaloffsetend}{-.5ex}
\makeatletter
\newcounter{ALG@tikzmark@tempcnta}

\newcommand\ALG@tikzmark@start{%
    \global\let\ALG@tikzmark@last\ALG@tikzmark@starttext%
    \expandafter\edef\csname ALG@tikzmark@\theALG@nested\endcsname{\theALG@tikzmark@tempcnta}%
    \tikzmark{ALG@tikzmark@start@\csname ALG@tikzmark@\theALG@nested\endcsname}%
    \addtocounter{ALG@tikzmark@tempcnta}{1}%
}

\def\ALG@tikzmark@starttext{start}
\newcommand\ALG@tikzmark@end{%
    \ifx\ALG@tikzmark@last\ALG@tikzmark@starttext
    \else
        \tikzmark{ALG@tikzmark@end@\csname ALG@tikzmark@\theALG@nested\endcsname}%
        \tikz[overlay,remember picture] \draw[\ALGtikzmarkcolor] let \p{S}=($(pic cs:ALG@tikzmark@start@\csname ALG@tikzmark@\theALG@nested\endcsname)+(\ALGtikzmarkextraindent,\ALGtikzmarkverticaloffsetstart)$), \p{E}=($(pic cs:ALG@tikzmark@end@\csname ALG@tikzmark@\theALG@nested\endcsname)+(\ALGtikzmarkextraindent,\ALGtikzmarkverticaloffsetend)$) in (\x{S},\y{S})--(\x{S},\y{E});%
    \fi
    \gdef\ALG@tikzmark@last{end}%
}

\apptocmd{\ALG@beginblock}{\ALG@tikzmark@start}{}{\errmessage{failed to patch}}
\pretocmd{\ALG@endblock}{\ALG@tikzmark@end}{}{\errmessage{failed to patch}}
\makeatother

\newcommand*{\thead}[1]{\bfseries{#1}}

\usepackage{comment}
\usepackage[utf8]{inputenc}

\renewcommand{\vec}[1]{\mbox{\boldmath $ #1 $}}

\newcommand{\norm}[1]{\left\lVert#1\right\rVert}

\newcommand{\OT}{\mathcal O}

\begin{document}

\pagestyle{headings}

\title{Iterative Implicit Methods for Solving Hodgkin-Huxley Type Systems}
\author{J\"urgen Geiser and Dennis Ogiermann}
\institute{Ruhr University of Bochum, \\
Institute of Theoretical Electrical Engineering, \\
Universit\"atsstrasse 150, D-44801 Bochum, Germany \\
\email{juergen.geiser@ruhr-uni-bochum.de}}
\maketitle

\begin{abstract}

We are motivated to approximate solutions of a Hodgkin-Huxley type model with implicit methods. 
As a representative we chose a psychiatric disease model containing stable as well as chaotic cycling behaviour.
We analyze the bifurcation pattern and show that some implicit methods help to preserve the limit cycles of such systems.
Further, we applied adaptive time stepping for the solvers to boost the accuracy, allowing us a preliminary zoom into the chaotic area of the system.

\end{abstract}



\section{Introduction}

We are motivated to model a nonlinear dynamic problem in neuroscience.
The most prominent system to describe the dynamics of neural cells is the Hodgkin Huxley model \cite{hodgkin1952}.
It is characteristic for this class of models to exhibit highly nonlinear oscillations in response to some external input \cite{chen2018}.
Sometimes we can observe chaotic oscillations, as for example in a small regime within the originally given parametrization of the Hodgkin and Huxley's model \cite{guck2002}.
Many subsequent models for biological oscillators have been either derived from this system or inspired by it. 
For details see \cite{post2016} and \cite{izhikevich2003}.

To study such delicate nonlinear dynamics, it is important to deal with stiff ODE solvers, which preserve the structure of the solution, see \cite{chen2018} and \cite{hairer2002}.
Based on the high quality of explicit and implicit time-integrators, which can be combined with conservation scheme, see \cite{trofimov2009}, we propose novel semi-implicit iterative methods, see \cite{geiser2015}.

The paper is outlined as follows. The model is introduced in Section \ref{modell}.
In Section \ref{methods}, we discuss the different numerical methods and present the
convergence analysis. The numerical experiments are done in Section \ref{numerics} and
the conclusion is presented in Section \ref{concl}.

\section{Mathematical Model}
\label{modell}
The classical Hodgkin-Huxley model is a parabolic partial differential equation with nonlinear reaction parts, see \cite{hodgkin1952}.
It models the dynamic behaviour of the the giant squid axon, which is a part of a neural cell.
Neural cells transfer information with the help of voltage peaks (so called action-potentials).
The voltage peaks base on the imbalance of the inner and outer ions and their diffusion, which is controlled by the potential difference across the cell's membrane.
The involved ions and channels are dependent on the type of neuron.

The standard Hodgkin-Huxley model is based on the flux of $Na^+$ and $K^+$ ions trough ion channels in the cell's membrane and proton pumps to provide a non-equilibrium environment.
Proton pumps move $Na^+$ ions out and the $K^+$ in by consuming ATP, forcing an imbalance of $Na^+$ and $K^+$ ions in the extracellular space (cell's outside) and intracellular space (cell's inside) respectively.
The activation state of these ion channels is controlled by voltage (potential) at the membrane. 
When enough voltage is present, then the fast $Na^+$ channels start to open, launching a diffusion-driven inflow of $Na^+$ ions from the extracellular space outside into the inner-cell, changing the cell's membrane potential towards positive values.
After a short time the $Na^+$ channels close and keep this closed state over a short time (they are called to be refractory). 
Slow $K^+$ channels open delayed to the fast $Na^+$ channels, such that we have an outflow of $K^+$ ions, parallel to the closing $Na^+$ channels. 
This mechanism again introduces a change in the potential, turning it back to the initial potential.
Putting these ideas together and taking into consideration, that the surface of neurons is geometrically rather complex and that the ion channels are not perfectly equal distributed over this surface, we can derive a partial differential equation to describe these spikes along an axon (see \cite{hodgkin1952} or \cite{izhikevich2003}):
\begin{align*}
C\frac{\partial V}{\partial t} &= I + D_m \frac{\partial^2 V}{\partial x^2} - \overbrace{\overline{g}_K n^4 (V - E_K)}^{I_K} - \overbrace{\overline{g}_{Na} m^3 h (V - E_{Na})}^{I_{Na}} - \overbrace{g_{L} (V - E_{L})}^{I_L} , \\
\frac{dn}{dt} &= \frac{n_\infty (V) - n}{\tau_n(V)} , \\
\frac{dm}{dt} &= \frac{m_\infty (V) - m}{\tau_m(V)} , \\
\frac{dh}{dt} &= \frac{h_\infty (V) - h}{\tau_h(V)} ,
\end{align*}
where $ D_m \frac{\partial^2 V}{\partial x^2} $ is the longitudinal conductivity.
$I_{Na}$ are $I_{K}$ are the natrium and kalium induced currents, $I_{L}$ is the leak current and $I$ is some externally applied current.
$n$ models the slow $K^+$ channel activation, while $m$ and $h$ describe the $Na^+$ channel activation and inactivation.
All parameters can be determined experimentally.

Further we have
\begin{align*}
n_\infty (V) &= \frac{\alpha_n (V)}{\alpha_n (V) + \beta_n (V)} , &
\tau_n (V) &= \frac{1}{\alpha_n (V) + \beta_n (V)} , \\
m_\infty (V) &= \frac{\alpha_m (V)}{\alpha_m (V) + \beta_m (V)} , &
\tau_m (V) &= \frac{1}{\alpha_m (V) + \beta_m (V)} , \\
h_\infty (V) &= \frac{\alpha_h (V)}{\alpha_h (V) + \beta_h (V)} , &
\tau_h (V) &= \frac{1}{\alpha_h (V) + \beta_h (V)} ,
\end{align*}
and the transition rates:
\begin{align*}
\alpha_n (V) &= 0.01 \frac{10-V}{exp \left( \frac{10-V}{V} \right) - 1} , &
\beta_n (V) &= 0.125 exp \left( \frac{-V}{80} \right) , \\
\alpha_m (V) &= 0.1 \frac{25 - V}{exp \left( \frac{20-V}{10} \right) - 1} , &
\beta_m (V) &= 4 exp \left( \frac{-V}{18} \right) , \\
\alpha_h (V) &= 0.07 exp \left( \frac{-V}{20} \right) , &
\beta_h (V) &= exp \left(\frac{30 - V}{10} \right) + 1 .
\end{align*}

This system is original Hodgkin-Huxley PDE \cite{hodgkin1952}.
Now assuming an ideal model of a neuron (more specifically its axon) as a cable, such that the spatial sizes are homgeneous and independent, we can reduce the model to a system of ordinary differential equations of the 
form (see also \cite{hodgkin1952}):
\begin{align*}
C\frac{dV}{dt} &= I - \overbrace{\overline{g}_K n^4 (V - E_K)}^{I_K} - \overbrace{\overline{g}_{Na} m^3 h (V - E_{Na})}^{I_{Na}} - \overbrace{g_{L} (V - E_{L})}^{I_L} , \\
\frac{dn}{dt} &= \frac{n_\infty (V) - n}{\tau_n(V)} , \\
\frac{dm}{dt} &= \frac{m_\infty (V) - m}{\tau_m(V)} , \\
\frac{dh}{dt} &= \frac{h_\infty (V) - h}{\tau_h(V)} .
\end{align*}

There also exist model-reductions of the HH model, mostly based on 2D ODEs (e.g. see \cite{izhikevich2007,post2016}).
One of the most famous one is the FitzHugh-Nagumo (FHN) model \cite{fitzhugh1969,nagumo1962}. 
Such models cannot show chaotic behaviour as a consequence of the Poincare-Bendixson theorem \cite{hirsch1974}.
The FHN model can also be interpreted as a generalisation of the Van-der-Pol Systems and is given as:
\begin{align*}
    \dot{V} &= V(a-V)(V-1)-w+I , \\
    \dot{w} &= bV - cw .
\end{align*}

\subsection{Hodgkin-Huxley Type Models}
To the best of our knowledge there exists no formal definition of which models exactly belong the class of Hodgkin-Huxley type systems.
Informally we refer to Hodgkin-Huxley type systems as differential equations as a special class of potentially nonlinear oscillating systems, where oscillations of an observable quantity are induced by the interplay with some independent but dynamic activations.
Closest to a definition of this class is the {\it generalized deterministic Hodgkin-Huxley equation} by Tim Austin \cite{austin2008}.
Based on the definition given from \cite{austin2008} and observations we propose the following definition for the class of Hodgkin-Huxley type (HHT) systems
\begin{align}
    \label{eq:generalization:hht-pde-o}
    \tau_o(o) \frac{d o}{d t} =& \nabla \cdot (D\nabla o) + f_o(o, {\bf a}) +  I,  \\ 
    \label{eq:generalization:hht-pde-i}
    \tau_i(o) \frac{d a_i}{d t} =& f_i(a_i, o) \quad \forall i \in \{1, \dots, n\},
\end{align}
which can be interpreted as a special case of reaction-diffusion systems.

Here $o$ describes an observable quantity, $I$ describes some external input function and $a_i$ are activation quantities. 
$f_o$ couples the observed quantity to the activations and may contain partial differential and integral operators, while the $f_i$'s describe analytic couplings of the activation back with the observable quantity.
From a modeling perspective we can sometimes an ideal case, where the spatial domain of PDE (\ref{eq:generalization:hht-pde-o}) is homgeneous and independent, such that the system reduces to the following ODE:
\begin{align}
    \label{eq:generalization:hht-ode-o}
    \tau_o(o) \frac{d o}{d t} =& f_o(o, {\bf a}) +  I,  \\  
    \label{eq:generalization:hht-ode-a}
    \tau_i(o) \frac{d a_i}{d t} =& f_i(a_i, o) \quad \forall i \in \{1, \dots, n\}.
\end{align}
This way our system contains naturally well-known ODEs used to model neural dynamics.
To the best of our knowledge this preliminary definition contains most systems which has been attributed as Hodgkin-Huxley typed in scientific literature so far.
We are aware that a special class of systems is not directly captured, the hybrid dynamical system models, since their solutions are discontinuous (to speed up computations of trajectories) \cite{izhikevich2010}, although the underlying continuous part of the system is.

\subsection{A Hodgkin-Huxley Type Nonlinear Disease Dynamics Model}

Trough this paper we deal with a HHT model appearing in neural modeling from neuroscience \cite{huber1998} and as the deterministic part for a stochastic disease model in neuropsychiatry \cite{huber2004}.
We chose this model as a representative system for the class of HHT models because it exhibits a rich amount behaviour in response to a constant input.
The model is given by the following system of ordinary differential equations
\begin{equation}
\label{eq:disease-model} 
\left.
\begin{aligned}
\tau_x \frac{d x}{d t} &= - x - \sum_{i \in \{he, li, le\}} a_i w_i (x -x_i) - a_{hi}^2 w_{hi} (x - x_{hi}) + S \\
\tau_i \frac{d a_i}{d t} &= F_i(x) - a_i \quad  \forall i \in \{ he, hi, le, li\}
\end{aligned}
\quad \right \}
\end{equation}
clearly fitting in the HHT class defined in the equations (\ref{eq:generalization:hht-ode-o}-\ref{eq:generalization:hht-ode-a}). 
We have x as an observable, where peaks represent events within the disease.
Further $\{ he, hi, le, li\}$ are the different activation types, operating on two time scales.
Elements starting with $h$ describe the fast time scale and model a high activation threshold, while elements starting with $l$ describe the slow time scale with low activation threshold respectively.
$e$ describes an excitatory and $i$ a corresponding inhibitory quantity. 
$F_i$ are sigmoidal functions of the form
\begin{align*}
 F_i(x) = \frac{1}{ 1 + \exp(- \Delta_i (x - \tilde{x}_{i}) )},
\end{align*}
where $\tilde{x}_{i}$ is the half-activation levels and $\Delta_i$ is the steepness of the sigmoidal function.
The fast excitatory quantity is assumed to activate instantaneously, so the model always has $\tau_{he} = 0$, implying $a_{he} = F_{he}(x)$.
As a consequence we also reduced the dimension of our dynamical system from 5 to 4.

\section{Numerical Methods}
\label{methods}

In the following we apply and discuss composition methods, as well as structure preserving methods based on finite difference and iterative schemes, which are also known to be successful in approximating solutions for various reaction-diffusion type equations.
We restrict us to composition methods, while also in the literature, there exists different other types of solver methods, e.g., tailored multi-step methods
as the Rush-Larson method, see \cite{perego2009}.

For notational simplicity we assume S to be time-independent such that the system gets autonomous. 
Further we introduce the following notation:
\begin{itemize}
\item ${\bf u} = (x, a_{hi}, a_{le}, a_{li})^T = (x, {\bf a})^T$ is the exact solution,
\item ${\bf u}^n = \left(x(t^n), a_{hi}(t^n), a_{le}(t^n), a_{li}(t^n)\right)^T = (x(t^n), {\bf a}^n)^T$ is defined as the solution at the time-point $t^n$,
\item Analogously ${\bf u}_i^{n} = (x_i(t^n), {\bf a}^n_i)^T$ is defined as the iterative solution of ${\bf u}$ in the $i$-th iterative step at the time-point $t^n$.
\end{itemize}
Bold letters indicate vectorial objects and $( \;\cdot\;)^T$ is the transpose.

\subsection{Composition with respect to Hamiltonian Systems}

If we apply a Van-der-Pol oscillator, which is a very simple Hodgkin-Huxley type system, we can reformulate the oscillator with respect to the non-stiff case into a Hamiltonian system and apply splitting approaches for the Hamiltonian systems.
The Van-der-Pol oscillator is given as:
\begin{align*}
    \frac{d x_1}{dt} &= x_2 , \\
     \frac{d x_2}{dt} &= \mu (1 - x_1^2) x_2 - x_1 ,
\end{align*}
where for $\mu = 0$, we obtain the harmonic oscillator with
the Hamiltonian system 
$$
H(x_1, x_2) = \frac{1}{2} (x_1^2 + x_2^2),
$$
although also other approaches are possible to uncover the systems hamiltonian \cite{shah2015}.
With these structural observations the idea is to apply such  composition methods, which are
known for the Hamiltonian system, i.e. Semi-implicit Euler scheme and St\"ormer-Verlet scheme \cite{hairer2002,hairer2003}), which are symplectic schemes if they are applied to a Hamiltonian system.

We introduce the following composition in operator notation for the disease model:
\begin{eqnarray}
\label{eq:disease-model:additive-splitting}
\frac{d {\bf u}}{d t} &&= {\bf F}({\bf u}) + {\bf S} = {\bf F}_1({\bf u}) + {\bf F}_2({\bf u}) + {\bf S},
\end{eqnarray}
where
\begin{eqnarray*}
&& {\bf F}_1({\bf u}) = \left( \begin{array}{c}
 \frac{- x - \sum_{i \in \{he, li, le\}} a_i w_i (x -x_i) - a_{hi}^2 w_{hi} (x - x_{hi})}{\tau_x},
      0,
      0,
      0  
      \end{array} \right)^T, \\
&& {\bf F}_2({\bf u}) = \left( \begin{array}{c}
       0,
       \frac{F_{hi}(x) - a_{hi}}{ \tau_{hi}},
       \frac{F_{le}(x) - a_{le}}{ \tau_{le}},
       \frac{F_{li}(x) - a_{li}}{ \tau_{li}}
       \end{array} \right), \;
{\bf S} = \left( \begin{array}{c}
      \frac{S}{\tau_x},
      0,
      0,
      0
 \end{array} \right)^T.
\end{eqnarray*}
Basing on this we define
\begin{align*}
f_1(x, {\bf a}) &= \frac{- x - \sum_{i \in \{he, li, le\}} a_i w_i (x -x_i) - a_{hi}^2 w_{hi} (x - x_{hi})}{\tau_x}, \\
{\bf f}_2(x, {\bf a}) &= \left( \frac{F_{hi}(x) - a_{hi}}{ \tau_{hi}}, \frac{F_{le}(x) - a_{le}}{ \tau_{le}}, \frac{F_{li}(x) - a_{li}}{ \tau_{li}} \right)^T,
\end{align*}
such that the algorithms are given as:
\begin{itemize}
\item Semi-implicit Euler scheme:
\begin{equation}
\label{eq:sie-scheme}
\left.
\begin{aligned}
x^{n+1} &= x^n + \Delta t \; f_1(x^n, {\bf a}^n) + \Delta t \; \frac{S}{\tau_x} \\
{\bf a}^{n+1} & = {\bf a}^n +  \Delta t \; {\bf f}_2(x^{n+1}, {\bf a}^{n+1})
\end{aligned}
\qquad \qquad \qquad \qquad \right \} 
\end{equation}
\item Störmer-Verlet scheme:
\begin{equation}
\label{eq:sv-scheme}
\left.
\begin{aligned}
x^{n+1/2} &= x^n + \frac{\Delta t}{2} \; f_1(x^n, {\bf a}^n) + \frac{\Delta t}{2} \; S, \\
{\bf a}^{n+1} &= {\bf a}^{n} +  \Delta t \; {\bf f}_2(x^{n+1/2}, {\bf a}^{n+1}) , \\
x^{n+1} &= x^{n+1/2} + \frac{\Delta t}{2} \; f_1(x^{n+1/2}, {\bf a}^{n+1}) + \frac{\Delta t}{2} \; \frac{S}{\tau_x}
\end{aligned}
\qquad \right \} 
\end{equation}
%

%
%

\end{itemize}

\begin{remark}
We can solve equations depending explicit on an $x$ and implicit on ${\bf a}$ directly, since the equations can be trivially rearranged on account of the linearity and independence on ${\bf a}$ in ${\bf f}_2$.
\end{remark}

\begin{remark}
For the semi-implicit Euler we have a global convergence order of $\OT(\Delta t)$ and for the St\"ormer-Verlet $\OT(\Delta t^2)$.
\end{remark}

\subsection{Iterative Schemes Based on Finite Difference Schemes}

We deal with the disease model, which is given as:

\begin{equation}
\label{eq:ode-hhm}
\left.
\begin{aligned}
&& \frac{d {\bf u}}{dt} = {\bf F}({\bf u})\\
&& {\bf u}(0) = {\bf u}^0
\end{aligned}
\qquad \right \} 
\end{equation}
%


%
%

We assume to deal with a system containing exactly one periodic orbit (in properly parameterized regime).
This implies there exists a $\tilde{t}>0$ such that for all points ${\bf u}_0$ starting on this orbit holds:
\[
    \norm{{\bf u}(0) - {\bf u}(\tilde{t})} = 0
\]
We call the smallest $\tilde{t}$ the period of an orbit.
We apply a semi-impicit Crank-Nicolson scheme (CN), see also \cite{trofimov2009},
which is conservative and given as:
\begin{equation}
\label{eq:cn-scheme}
{\bf u}^{n+1} =  {\bf u}^{n} + \frac{\Delta t}{2} \; \left( {\bf F}({\bf u}^{n+1})   + {\bf F}({\bf u}^{n}) \right) 
\end{equation}
%

Here, we have a nonlinear equation system, which have to apply additional nonlinear solvers, e.g. Newton's method.
Therefore, we propose iterative schemes, which embed via iterative step to the semi-implicit structures.

\begin{remark}
The semi-implicit CN method can be derived via operator-splitting approach:
\begin{eqnarray*}
\label{cfds_1_1}
&& \tilde{{\bf u}}^{n+1} =  {\bf u}^{n} + \frac{\Delta t}{2} \; {\bf F}({\bf u}^{n}) , \\
\label{cfds_1_2}
&& {\bf u}^{n+1} =  \tilde{{\bf u}}^{n+1} + \frac{\Delta t}{2} \; {\bf F}({\bf u}^{n+1}) ,
\end{eqnarray*}
where the first equation (\ref{cfds_1_1}) is explicit and can be done directly, the second one (\ref{cfds_1_2}) is implicit and solved with a fixpoint scheme as:
\begin{eqnarray*}
\label{cfds_1_3}
&& {\bf u}_i^{n+1} =  \tilde{{\bf u}}^{n+1} + \frac{\Delta t}{2} \; {\bf F}({\bf u}_{i-1}^{n+1}) , 
\end{eqnarray*}
where the starting condition is ${\bf u}_0^{n+1} = {\bf u}^{n}$ and we apply
$i = 1, \ldots, I$, while $I$ is an integer and we stop if we have the 
error bound $\norm{ {\bf u}_{i}^{n+1} - {\bf u}_{i-1}^{n+1}  }\le \varepsilon$ with $\varepsilon$ as an
error bound.
\end{remark}

\subsubsection{Semi-implicit Integrators}

In the following, we deal with semi-implicit integrators.
We introduce the following the following convention for intermediate results:
%
\begin{itemize}
\item We initialize the iterative scheme with the solution in time point $t^{n}$, i.e. ${\bf u}_0^{n+1} = {\bf u}^{n}$. 
\item We set the approximation for the next time point $t^{n+1}$ with the iterative solution in the $i$-th iterative step, i.e. ${\bf u}^{n+1} = {\bf u}_i^{n+1}$
 \item We will denote the splitting from equation (\ref{eq:disease-model:additive-splitting}) as follows:\\
$${\bf F}({\bf u}, {\bf v}) := {\bf F}_1({\bf u}) + {\bf F}_2({\bf v}) + {\bf S}$$
\end{itemize}

%
We compute the approximations ${\bf u}(t^n)$ at the time points  $n = 1, 2, 3, \ldots, N$ coupled with a fixed-point iteration, where $t^N = T$. 
The initialization of the iterative scheme is given with the initial condition of the equations (\ref{eq:ode-hhm}) as $u^{0,1} = u^0$. 
For now the time step is defined as $\Delta t := t^{n} - t^{n-1}$, while the error bound is given as 
$\varepsilon$.
Based on this information we define the first three solvers with algorithms (\ref{algo:isie}-\ref{algo:isv}).
\begin{algorithm}[h!]
\caption{Iterative Semi-implicit Euler (ISIE)}
\label{algo:isie}

\begin{algorithmic}[1]
    \Require Initial solution $u^0$, time step $\Delta t$, max time $T$, tolerance $\varepsilon$, max iterations $I$
    \Ensure Approximation $u(0), u(t^1), \dots , u(T)$ 
    \State $n \leftarrow 0$
    \Repeat
        \State ${\bf u}_0^{n+1} \leftarrow {\bf u}^{n}$, $i \leftarrow 0$
        \Repeat
            \State $i \leftarrow i + 1$
            \State ${\bf u}_i^{n+1} \leftarrow {\bf u}^n + \Delta t \; {\bf F}({\bf u}_{i-1}^{n+1}, {\bf u}_i^{n+1})$ \Comment{equations (\ref{eq:sie-scheme})}
        \Until{$i = I$ or $ \norm{ {\bf u}_i^{n+1} - {\bf u}_{i-1}^{n+1} } \le \varepsilon$} \Comment{stopping criterion}
        \State ${\bf u}^{n+1} \leftarrow {\bf u}_i^{n+1}$, $n \leftarrow n + 1$
    \Until{$n \Delta t > T$} \Comment{termination criterion}
\end{algorithmic}
\end{algorithm}

\begin{algorithm}[h!]
\caption{Iterative Crank-Nicolson (ICN)}
\label{algo:icn}
\begin{algorithmic}[1]
    \Require Initial solution $u^0$, time step $\Delta t$, max time $T$, tolerance $\varepsilon$, max iterations $I$
    \Ensure Approximation $u(0), u(t^1), \dots , u(T)$ 
    \State $n \leftarrow 0$
    \Repeat
        \State ${\bf u}_0^{n+1} \leftarrow {\bf u}^{n}$, $i \leftarrow 0$
        \Repeat
            \State $i \leftarrow i + 1$
            \State ${\bf u}_i^{n+1} \leftarrow {\bf u}^n + \frac{\Delta t}{2} \; \left( {\bf F}({\bf u}_{i-1}^{n+1}, {\bf u}^{i, n+1}) +  {\bf F}({\bf u}^{n}, {\bf u}^{n}) \right)$ \Comment{equations (\ref{eq:cn-scheme})}
        \Until{$i = I$ or $ \norm{ {\bf u}_i^{n+1} - {\bf u}_{i-1}^{n+1} } \le \varepsilon$} \Comment{stopping criterion}
        \State ${\bf u}^{n+1} \leftarrow {\bf u}_i^{n+1}$, $n \leftarrow n + 1$
    \Until{$n \Delta t > T$} \Comment{termination criterion}
\end{algorithmic}
\end{algorithm}

\begin{algorithm}[h!]
\caption{Iterative Störmer-Verlet (ISV)}
\label{algo:isv}
\begin{algorithmic}[1]
    \Require Initial solution $u^0$, time step $\Delta t$, max time $T$, tolerance $\varepsilon$, max iterations $I$
    \Ensure Approximation $u(0), u(t^1), \dots , u(T)$ 
    \State $n \leftarrow 0$
    \Repeat
        \State ${\bf u}_0^{n+1} \leftarrow {\bf u}^{n}$, $i \leftarrow 0$
        \Repeat
            \State $i \leftarrow i + 1$
            \State $x^{n+1/2} \leftarrow x^n + \frac{\Delta t}{2} f_1(x^n, {\bf a}^n) + \frac{\Delta t}{2} S$
            \State ${\bf a}^{n+1} \leftarrow {\bf a}^{n} +  \Delta t \; {\bf f}_2(x^{n+1/2}, {\bf a}^{n+1})$ \Comment{equations (\ref{eq:sv-scheme})}
            \State $x^{n+1} \leftarrow x^{n+1/2} + \frac{\Delta t}{2} \; f_1(x^{n+1/2}, {\bf a}^{n+1}) + \frac{\Delta t}{2} \; \frac{S}{\tau_x}$
        \Until{$i = I$ or $ \norm{ {\bf u}_i^{n+1} - {\bf u}_{i-1}^{n+1} } \le \varepsilon$} \Comment{stopping criterion}
        \State ${\bf u}^{n+1} \leftarrow {\bf u}^{i,n+1}$, $n \leftarrow n + 1$
    \Until{$n \Delta t > T$} \Comment{termination criterion}
\end{algorithmic}
\end{algorithm}

\begin{remark}
The semi-implicit CN scheme based on the iterative approach is asymptotical conservative \cite{geiser2019}.
\end{remark}

Further we define two multipredictor multicorrector methods with algorithms (\ref{algo:mmrk4}) and (\ref{algo:irk4}).

\begin{algorithm}[h!]
\caption{Multipredictor Multicorrector Runge-Kutta-4 (MMRK4)}
\label{algo:mmrk4}
\begin{algorithmic}[1]
    \Require Initial solution $u^0$, time step $\Delta t$, max time $T$, tolerance $\varepsilon$, max iterations $I$
    \Ensure Approximation $u(0), u(t^1), \dots , u(T)$ \State $n \leftarrow 0$
    \Repeat
        \State ${\bf \tilde{u}}^{n+\frac{1}{2}} \leftarrow {\bf u}^n + \frac{\Delta t}{2} {\bf F}({\bf u}^n )$ \Comment{predictor (forward Euler)}
        \State ${\bf \hat{u}}^{n+\frac{1}{2}} \leftarrow {\bf u}^{n} + \frac{\Delta t}{2} \; {\bf F}( {\bf \tilde{u}}^{n+\frac{1}{2}})$ \Comment{corrector (backward Euler)}
        \State ${\bf \tilde{u}}^{n+1} \leftarrow {\bf u}^{n} + \Delta t \; {\bf F} ({\bf \hat{u}}^{n+\frac{1}{2}})$ \Comment{predictor (midpoint rule)}
        \State ${\bf u}^{n+1} \leftarrow {\bf u}^n + \frac{\Delta t}{6} \left( {\bf F}( {\bf u}^n) + 2 {\bf F}( {\bf \tilde{u}}^{n+\frac{1}{2}}) + 2 {\bf F}( {\bf \hat{u}}^{n+\frac{1}{2}}) + {\bf F}( {\bf \tilde{u}}^{n+1}) \right)$ \Comment{corrector (Simpson rule)}
        \State $n \leftarrow n + 1$
    \Until{$n \Delta t > T$} \Comment{termination criterion}
\end{algorithmic}
\end{algorithm}

\begin{algorithm}[h!]
\caption{Iterative Runge-Kutta-4 (IRK4)}
\label{algo:irk4}
\begin{algorithmic}[1]
    \Require Initial solution $u^0$, time step $\Delta t$, max time $T$, tolerance $\varepsilon$, max iterations $I$ and $J$
    \Ensure Approximation $u(0), u(t^1), \dots , u(T)$ \State $n \leftarrow 0$
    \Repeat
        \State ${\bf u}_0^{n+1} \leftarrow {\bf u}^{n}$, $i \leftarrow 0$, $j \leftarrow 0$
        \Repeat
            \State $i \leftarrow i + 1$
            \State ${\bf \tilde{u}}_i^{n+\frac{1}{2}} = {\bf u}^n + \frac{\Delta t}{4} \left ( {\bf F}({\bf u}^n ) +  {\bf F}({\bf \tilde{u}}_{i-1}^{n+\frac{1}{2}} ) \right)$ \Comment{predictor (Crank-Nicolson)}
        \Until{$i = I$ or $ ||{ {{\bf\tilde{u}}}_i^{n+\frac{1}{2}} - {\bf \tilde{u}}_{i-1}^{n+\frac{1}{2}}}|| \le \varepsilon$} \Comment{stopping criterion}
        \Repeat
            \State $j \leftarrow j + 1$
            \State ${\bf u}_j^{n+1} = {\bf u}^n + \frac{\Delta t}{6} \left( {\bf F}({\bf u}^n) + 4 {\bf F}( {\bf\tilde{u}}_i^{n+\frac{1}{2}}) + {\bf F}({\bf u}_{j-1}^{n+1}) \right)$ \Comment{corrector (Simpson rule)}
        \Until{$j = J$ or $ \norm{ {\bf u}_i^{n+1} - {\bf u}_{i-1}^{n+1} } \le \varepsilon$} \Comment{stopping criterion}
        \State ${\bf u}^{n+1} \leftarrow {\bf u}_j^{n+1}$, $n \leftarrow n + 1$
    \Until{$n \Delta t > T$} \Comment{termination criterion}
\end{algorithmic}
\end{algorithm}

\FloatBarrier
\subsection{Adaptive Time Step Control of the Iterative CN Scheme}

To improve the numerical results in the critical time-scales (i.e. the stiff parts of the evolution equation) we apply adaptive time step approaches.
We define the following norms:

\begin{itemize}
\item Absolute norm:
\begin{eqnarray}
\norm{{\bf u}^n} = \sqrt{x(t^n)^2 + a_{he}(t^n)^2 + a_{li}(t^n)^2 +  a_{le}(t^n)^2}
\end{eqnarray}
%
\item Maximum-norm:
%
\begin{eqnarray}
\norm{{\bf u}^n}_{max} = \max \left\{ |x(t^n)|, |a_{he}(t^n)|, |a_{li}(t^n)|, |a_{le}(t^n)| \right\}
\end{eqnarray}
\end{itemize}

The relative error is given as:
\begin{eqnarray}
e(t^n) = \frac{\norm{{\bf u}^{n+1} -{\bf u}^n  }}{\norm{ {\bf u}^{n+1} }} .
\end{eqnarray}

\subsubsection{PID-Controller}
We apply the following simple error-estimate (see \cite{kuzmin2011}), where we compute the time step for a given tolerance $\varepsilon$ at a timepoint $t^n$:
\begin{eqnarray}
\label{eq:pid-controller}
\Delta t^{n+1} = \left(\frac{e(t^{n-1})}{e(t^n)}\right)^{k_P}  \left(\frac{\varepsilon}{e(t^n)}\right)^{k_I}  \left(\frac{e^2(t^{n-1})}{e(t^n) e(t^{n-2})}\right)^{k_D} \Delta t^{n} ,     
\end{eqnarray}
where we assume the emprical PID (Proportional-Integral-Differential) parameters $k_P = 0.075, k_I = 0.175, \; k_D = 0.01$.
For the initialisation, means for $n=1$, we only apply the $I$ part, while for $n=2$ we apply the $I$ and $P$ part and for all later time steps (where we have all the parts $e(t^{n-2}), e(t^{n-1}), e(t^{n-2})$), we apply $I, P, D$.

\begin{algorithm}[h!]
\begin{algorithmic}[1]
\label{algo:pidicn}
\caption{Proportional-Integral-Differential-Controlled Iterative Crank-Nicolson (PIDICN)}
\Require Initial solution $u^0$, initial time step $\Delta t^0$, max time $T$, fixed-point iteration tolerance $\varepsilon_{fp}$, time controller tolerance $\varepsilon_{t}$, max iterations $I$ and $J$
    \Ensure Approximation $u(0), u(t^1), \dots , u(T)$ 
    \State $n \leftarrow 0$
    \State $\Delta t \leftarrow \Delta t^0$
    \Repeat
        \State ${\bf u}_0^{n+1} \leftarrow {\bf u}^{n}$, $i \leftarrow 0$
        \Repeat
            \State $i \leftarrow i + 1$
            \State ${\bf u}_i^{n+1} \leftarrow {\bf u}^n + \frac{\Delta t}{2} \; \left( {\bf F}({\bf u}_{i-1}^{n+1}, {\bf u}_i^{n+1}) +  {\bf F}({\bf u}^{n}, {\bf u}^{n}) \right)$ \Comment{equations (\ref{eq:cn-scheme})}
        \Until{$i = I$ or $ \norm{ {\bf u}_i^{n+1} - {\bf u}_{i-1}^{n+1} } \le \varepsilon_{fp}$} \Comment{stopping criterion}
        \State $\Delta t \leftarrow \left(\frac{e(t^{n-1})}{e(t^n)}\right)^{k_P}  \left(\frac{\varepsilon_{t}}{e(t^n)}\right)^{k_I}  \left(\frac{e^2(t^{n-1})}{e(t^n) e(t^{n-2})}\right)^{k_D} \Delta t$ \Comment{equation (\ref{eq:pid-controller})}
        \State ${\bf u}^{n+1} \leftarrow {\bf u}_i^{n+1}$, $n \leftarrow n + 1$, $t^{n+1} \leftarrow t^n + \Delta t$
    \Until{$t^{n+1} > T$} \Comment{termination criterion}
\end{algorithmic}
\end{algorithm}

\clearpage
\subsubsection{Classical Time Step Controller for the ICN}

We apply an additional automatic time step control which is given as following with a two scale ansatz, where we compute an approximation via large step $\Delta t$ and compare the solution with m consecutive substeps of length $\frac{\Delta t}{m}$ to give another approximation, which should be close to the large step if the approximator is accurate enough, given the current time step.
Solutions are rejected until the time step is small enough, which implies the approximation error is smaller than some bound.
We apply the following time step controller for second order schemes:
\begin{equation}
\label{eq:classical-controller}
\Delta t^* = \sqrt{\varepsilon \frac{\Delta t^2 (m^2 -1)}{\norm{{\bf u}_{\Delta t} - {\bf u}_{m \Delta t} }}} ,
\end{equation}
where $\Delta t^*$ is the optimal time step while ${\bf u}_{\Delta t}$ is the approximation by applying m small time steps an ${\bf u}_{m\Delta t}$ is the solution of an equivalent length large time step.

\begin{algorithm}[h!]
\begin{algorithmic}[1]
\label{algo:aicn}
\caption{Adaptive Iterative Crank-Nicolson (AICN)}
\Require Initial solution $u^0$, initial time step $\Delta t^0$, max time $T$, fixed-point iteration tolerance $\varepsilon_{fp}$, time controller tolerance $\varepsilon_{t}$, max iterations $I$
    \Ensure Approximation $u(0), u(t^1), \dots , u(T)$ 
    \State $n \leftarrow 0$, $\Delta t^* \leftarrow \Delta t^0$, $\Delta t^* \leftarrow \Delta t^0$
    \Repeat
        \State ${\bf u}_0^{n+1} \leftarrow {\bf u}^{n}$, $i \leftarrow 0$, $\Delta t \leftarrow \Delta t^*$
        \Repeat
            \State $i \leftarrow i + 1$
            \State ${\bf u}_i^{n+1} \leftarrow {\bf u}^n + \frac{\Delta t}{2} \; \left( {\bf F}({\bf u}_{i-1}^{n+1}, {\bf u}_i^{n+1}) +  {\bf F}({\bf u}^{n}, {\bf u}^{n}) \right)$ \Comment{equations (\ref{eq:cn-scheme})}
        \Until{$i = I$ or $ \norm{ {\bf u}_i^{n+1} - {\bf u}_{i-1}^{n+1} } \le \varepsilon_{fp}$} \Comment{stopping criterion}
        \State Compute ${\bf v}_i^{n+1}$ by applying the previous loop m times with time step $\frac{\Delta t}{m}$
        \State $\Delta t^* \leftarrow \sqrt{ \varepsilon_t \frac{\Delta t^2 (m^2 -1)}{\norm{{\bf u}_i^{n+1} - {\bf v}_i^{n+1} }} } $ \Comment{equation (\ref{eq:classical-controller})}
        \If{$\Delta t \leq \Delta t^*$}  \Comment{Reject approximation until "good enough"}     
        \State ${\bf u}^{n+1} \leftarrow {\bf u}_i^{n+1}$, $n \leftarrow n + 1$, $t^{n+1} \leftarrow t^n + \Delta t$
        \EndIf
    \Until{$t^{n+1} > T$} \Comment{termination criterion}
\end{algorithmic}
\end{algorithm}

\subsection{Time Step Controller for the Runge-Kutta Methods}

We extend the multipredictor-multicorrector algorithm of order $4$, 
see Algorithm (\ref{algo:mmrk4}) and an iterative CN+Simpson-Rule of order $4$, see Algorithm (\ref{algo:irk4}):

\begin{lemma}

We deal with 4th order time-integrator methods with tolerance $\varepsilon$.
Further, we assume that we have a 4th order numerical solver, which is
give as  ${\bf u}(t + \Delta t) = A_{\Delta t} \; {\bf u}(t)$ and
${\bf u}(t)$ is the exact solution at time $t$. We apply the $|| \cdot||_p$-norm
as a given vector norm, e.g., in the Banach-space.
 
Then the adaptive time stepping is given as:
\begin{eqnarray}
\label{eq:optimal-timestep-4thorder}
\Delta t^* = \left( \varepsilon \frac{\Delta t^4 (m^4 -1)}{\norm{{\bf u}_{\Delta t} - {\bf u}_{m \Delta t} }_2} \right)^{1/4} .
\end{eqnarray}

\end{lemma}

\begin{proof}

We assume $\norm{{\bf u} - {\bf u}_{\Delta t}} = \varepsilon$, which is
a prescribed tolerance.

We apply 2 different time-steps:
\begin{itemize}
\item A single large time-step $\Delta t$ with:
\begin{eqnarray*}
{\bf u}_{\Delta t}(t^n) = {\bf u} + A_{\Delta t} {\bf u}(t^{n-1})  ,
\end{eqnarray*}
\item A multiple small time-step $\Delta t /m$ with:
\begin{eqnarray*}
{\bf u}_{\Delta t/m}(t^n) =  {\bf u} + A_{\Delta t/m}^m {\bf u}(t^{n-1}),
\end{eqnarray*}

\end{itemize}

The local truncation error is given as:
\begin{eqnarray*}
{\bf u}_{\Delta t} = {\bf u} + \Delta t^4 e({\bf u}) + \OT(\Delta t^6) , \\
{\bf u}_{\Delta t/m} = {\bf u} + (\Delta t/m)^4 e({\bf u}) + \OT(\Delta t^6) ,
\end{eqnarray*}
and we assume to have the approximation:
\[
\norm{\frac{\vec{u}(t^n) - \vec{u}_{\Delta t^*}(t^n)}{{\Delta t^*}^4(0-1)}}_2
\approx 
\norm{\frac{\vec{u}_{\Delta t}(t^n) - \vec{u}_{\Delta t/m}(t^n)}{\Delta t^4(1-m^4)}}_2
\]

which can be interpreted as a scaling of the error estimates.

Using the norm property we can now pull out the divisors:

\[
\frac{\norm{\vec{u}(t^n) - \vec{u}_{\Delta t^*}(t^n)}_2}{\norm{{\Delta t^*}^4(0-1)}_1}
\approx 
\frac{\norm{\vec{u}_{\Delta t}(t^n) - \vec{u}_{\Delta t/m}(t^n)}_2}{\norm{\Delta t^4(1-m^4)}_1}
\]

we can simplify the divisors:

\[
\frac{\norm{\vec{u}(t^n) - \vec{u}_{\Delta t^*}(t^n)}_2}{{\Delta t^*}^4}
\approx 
\frac{\norm{\vec{u}_{\Delta t}(t^n) - \vec{u}_{\Delta t/m}(t^n)}_2}{\Delta t^4(m^4-1)}
\]

we assumed $\norm{\vec{u}(t^n) - \vec{u}_{\Delta t^*}(t^n)}_2 = \varepsilon$, which is our error control, such that we obtain the following crude approximation:

\[
\frac{\varepsilon}{{\Delta t^*}^4}
\approx 
\frac{\norm{\vec{u}_{\Delta t}(t^n) - \vec{u}_{\Delta t/m}(t^n)}_2}{\Delta t^4(m^4-1)}
\Leftrightarrow \Delta t^* \approx \sqrt[4]{\frac{\Delta t^4(m^4-1)}{\varepsilon \norm{\vec{u}_{\Delta t}(t^n) - \vec{u}_{\Delta t/m}(t^n)}_2}}
\]

Then, the adaptive time stepping is given as:
\begin{eqnarray*}
\Delta t^* = \left( \varepsilon \frac{\Delta t^4 (m^4 -1)}{\norm{{\bf u}_{\Delta t} - {\bf u}_{\Delta t/m} }_2} \right)^{1/4} 
\end{eqnarray*}

\end{proof}

The improved automatic time step controlled $4$-th order methods are now given with algorithms (\ref{algo:ark4}) and (\ref{algo:airk4}).

\begin{algorithm}
\caption{Multipredictor Multicorrector Runge-Kutta-4 (ARK4)}
\label{algo:ark4}
\begin{algorithmic}[1]
\Require Initial solution $u^0$, initial time step $\Delta t^0$, max time $T$, time controller tolerance $\varepsilon_{t}$, max iterations $I$
    \Ensure Approximation $u(0), u(t^1), \dots , u(T)$ 
    \State $n \leftarrow 0$, $\Delta t^* \leftarrow \Delta t^0$, $\Delta t^* \leftarrow \Delta t^0$
    \Repeat
        \State $\Delta t \leftarrow \Delta t^*$
        \State ${\bf \tilde{u}}^{n+\frac{1}{2}} \leftarrow {\bf u}^n + \frac{\Delta t}{2} {\bf F}({\bf u}^n )$ \Comment{predictor (forward Euler)}
        \State ${\bf \hat{u}}^{n+\frac{1}{2}} \leftarrow {\bf u}^{n} + \frac{\Delta t}{2} \; {\bf F}( {\bf \tilde{u}}^{n+\frac{1}{2}})$ \Comment{corrector (backward Euler)}
        \State ${\bf \tilde{u}}^{n+1} \leftarrow {\bf u}^{n} + \Delta t \; {\bf F} ({\bf \hat{u}}^{n+\frac{1}{2}})$ \Comment{predictor (midpoint rule)}
        \State ${\bf u}^{n+1} \leftarrow {\bf u}^n + \frac{\Delta t}{6} \left( {\bf F}( {\bf u}^n) + 2 {\bf F}( {\bf \tilde{u}}^{n+\frac{1}{2}}) + 2 {\bf F}( {\bf \hat{u}}^{n+\frac{1}{2}}) + {\bf F}( {\bf \tilde{u}}^{n+1}) \right)$ \Comment{corrector (Simpson rule)}
        \State Compute ${\bf v}^{n+1}$ by applying the previous scheme m times with time step $\frac{\Delta t}{m}$
        \State $\Delta t^* \leftarrow \sqrt[4]{ \varepsilon_t \frac{\Delta t^4 (m^4 -1)}{\norm{{\bf u}^{n+1} - {\bf v}^{n+1} }} } $ \Comment{equation (\ref{eq:optimal-timestep-4thorder})}
        \If{$\Delta t \leq \Delta t^*$}  \Comment{Reject approximation until "good enough"}     
        \State ${\bf u}^{n+1} \leftarrow {\bf u}_i^{n+1}$, $n \leftarrow n + 1$, $t^{n+1} \leftarrow t^n + \Delta t$
        \EndIf
    \Until{$t^{n+1} > T$} \Comment{termination criterion}
\end{algorithmic}
\end{algorithm}

\begin{algorithm}
\caption{Adaptive Iterative Runge-Kutta-4 (AIRK4)}
\label{algo:airk4}
\begin{algorithmic}[1]
\Require Initial solution $u^0$, initial time step $\Delta t^0$, max time $T$,fixed-point iteration tolerance $\varepsilon_{fp}$, time controller tolerance $\varepsilon_{t}$, max iterations $I$
    \Ensure Approximation $u(0), u(t^1), \dots , u(T)$ 
    \State $n \leftarrow 0$, $\Delta t^* \leftarrow \Delta t^0$, $\Delta t^* \leftarrow \Delta t^0$
    \Repeat
        \State $\Delta t \leftarrow \Delta t^*$
        \State ${\bf u}_0^{n+1} \leftarrow {\bf u}^{n}$, $i \leftarrow 0$, $j \leftarrow 0$
        \Repeat
            \State $i \leftarrow i + 1$
            \State ${\bf \tilde{u}}_i^{n+\frac{1}{2}} = {\bf u}^n + \frac{\Delta t}{4} \left ( {\bf F}({\bf u}^n ) +  {\bf F}({\bf \tilde{u}}_{i-1}^{n+\frac{1}{2}} ) \right)$ \Comment{predictor (Crank-Nicolson)}
        \Until{$i = I$ or $ ||{ {{\bf\tilde{u}}}_i^{n+\frac{1}{2}} - {\bf\tilde{u}}}||_{i-1}^{n+\frac{1}{2}} \le \varepsilon_{fp}$} \Comment{stopping criterion}
        \Repeat
            \State $j \leftarrow j + 1$
            \State ${\bf u}_j^{n+1} = {\bf u}^n + \frac{\Delta t}{6} \left( {\bf F}({\bf u}^n) + 4 {\bf F}( {\bf\tilde{u}}_i^{n+\frac{1}{2}}) + {\bf F}({\bf u}_{j-1}^{n+1}) \right)$ \Comment{corrector (Simpson rule)}
        \Until{$j = J$ or $ \norm{ {\bf u}_i^{n+1} - {\bf u}_{i-1}^{n+1} } \le \varepsilon_{fp}$} \Comment{stopping criterion}
        \State Compute ${\bf v}_j^{n+1}$ by applying the previous scheme m times with time step $\frac{\Delta t}{m}$
        \State $\Delta t^* \leftarrow \sqrt[4]{ \varepsilon_t \frac{\Delta t^4 (m^4 -1)}{\norm{{\bf u}_j^{n+1} - {\bf v}_j^{n+1} }} } $ \Comment{equation (\ref{eq:optimal-timestep-4thorder})}
        \If{$\Delta t \leq \Delta t^*$}  \Comment{Reject approximation until "good enough"}     
        \State ${\bf u}^{n+1} \leftarrow {\bf u}_i^{n+1}$, $n \leftarrow n + 1$, $t^{n+1} \leftarrow t^n + \Delta t$
        \EndIf
    \Until{$t^{n+1} > T$} \Comment{termination criterion}
\end{algorithmic}
\end{algorithm}

\FloatBarrier
\section{Numerical Results}
\label{numerics}

Trough this section we present a short analysis of the dynamical system in combination with the performance of the in previous section derived solvers.
For the implementation we used Julia\footnote{\url{https://julialang.org/}} 1.1.
A Jupyter notebook containing the implementation of this section can be found online under \url{https://git.noc.ruhr-uni-bochum.de/ogierdst/solving-hodgkin-huxley-type-systems/}.

We deal with the disease dynamics model (\ref{modell}) and the parametrization taken from \cite{huber2004}:
\begin{eqnarray*}
\tau_x = 10, \; w_{hi} = 20, w_{he} = 15, \; w_{li} = 18, w_{le} = 3 , \\
 x_{le} = x_{li} = -30 , x_{he} = x_{hi} = 110 , \\
\tau_{hi} = 2, \; \tau_{he} = 0, \; \tau_{li} = 50, \; \tau_{le} = 10, \; \\
\Delta_{he} = \Delta_{hi} =  \Delta_{li} = \Delta_{le} = 0.25, \\
\tilde{x}_{le} = \tilde{x}_{li} = 20, \; \tilde{x}_{he} = \tilde{x}_{hi} = 35,
\end{eqnarray*}

Note that since $\tau_{he} = 0$ we obtain a reduced system of order 4, where $a_{he} = F_{he}(x)$.
This choice corresponds to an instantaneous activation of $a_{he}$, effectively reducing the system's dimension to 4.

\subsection{Exploring Structural Properties via Computational Bifurcation Analysis}
We start by exploring the system's overall behavior for varying $S \in [0,400]$.
This section is not ment to replace a rigorous dynamical system analysis, but to outline its coarse structure to ease the analysis of the solvers.
For convenience we use {\it Tsit5} from the JuliaDiffEq package \cite{juliadiffeq} as the solver when not otherwise stated.
This way we provide a tested baseline as a foundation to compare the implementation of our solvers to.

As a first step we extract the system's fixed-points, which are given by setting the change in all dimensions to zero.
Formally we first rewrite the model (\ref{modell})
$$
\frac{d {\bf u}}{d t} = {\bf f}({\bf u},S) ,
$$
and set it to zero, i.e.
$$
{\bf f}({\bf u^*},S) = {\bf 0}.
$$
Here ${\bf u^*}$ denotes a fixed point. 
The system's special structure allows us to reduce this problem to one dimension, as 
$$
\forall i \in \{ he, hi, le, li\}: 0 = F_i(x^*) - a_i \Longleftrightarrow a_i = F_i(x^*),
$$
which results in
\begin{align}
\label{eq:fixed-point-x}
 0 = -x^* - \left(\sum_{i \in \{he, le, li \}} F_i(x^*) \; w_{i} \; (x^* - x_{i}) \right) - F_{hi}(x^*)^2 \; w_{hi} \; (x^* - x_{hi}) + S.
\end{align}

It can be easily shown that this function is unbounded and strictly monotonically decreasing for our chosen parametrization.
This implies that there is a single fixed point for each S.
We obtain the corresponding $a_i^*$'s explicitly by plugging the solution back into the corresponding equations.
Approximating some fixed points with Newton-Raphson and linearising around these gives an idea of its stability properties. 
This yields the Jacobian $J_{ij} = \frac{\partial f_i}{\partial u_j} |_{{\bf u} = {\bf u}^*}$, which is explicitly:
$$
\begin{bmatrix}
-\frac{1 + \left( \Delta_{he} a^*_{he} (1-a^*_{he}) w_{he} (x^* - x_{he})   a^*_{he} w_{he} + {a^*_{hi}}^2 w_{hi} + a^*_{le} w_{le} + a^*_{li} w_{li} \right)}{\tau_x}
& -\frac{2 a^*_{hi} w_{hi} (x^* - x_{hi})}{\tau_x}
& -\frac{w_{le} (x^* - x_{le})}{ \tau_{x}}
& -\frac{w_{li} (x^* - x_{li})}{\tau_{x}}\\
\frac{\Delta_{hi}  a^*_{hi} (1 - a^*_{hi})}{ \tau_{hi} } & -\frac{1}{\tau_{hi}} & 0 & 0 \\
\frac{\Delta_{le} a^*_{le} (1 - a^*_{le})}{ \tau_{le}}  & 0 & -\frac{1}{ \tau_{le} }& 0 \\
\frac{\Delta_{li} a^*_{li} (1 - a^*_{li})}{ \tau_{li}} & 0 & 0 & -\frac{1}{ \tau_{li} }
\end{bmatrix}
$$
Note that $f_i$ is the disease models i-th equation while $F_i$ denotes the sigmoidal function for the corresponding activation.

\begin{figure}
    \centering
    \includegraphics[width=0.3\textwidth]{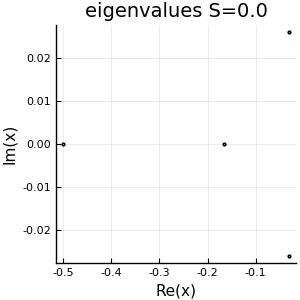}
    \includegraphics[width=0.3\textwidth]{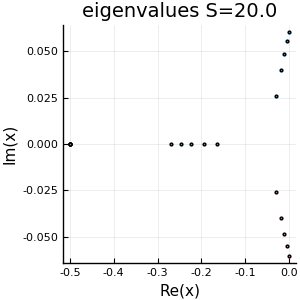}
    \includegraphics[width=0.3\textwidth]{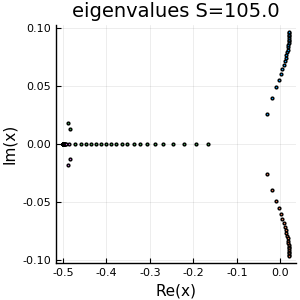}\\
    \includegraphics[width=0.3\textwidth]{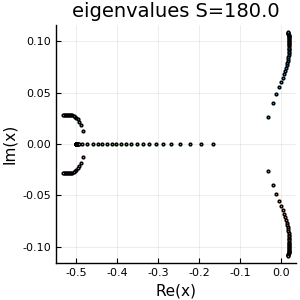}
    \includegraphics[width=0.3\textwidth]{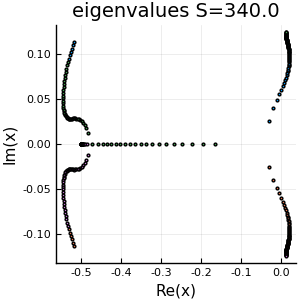}
    \includegraphics[width=0.3\textwidth]{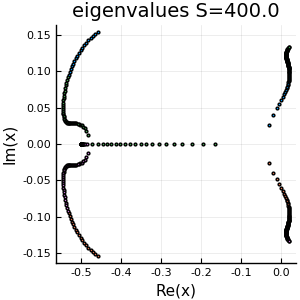}
    \caption{Evolution of the system's Jacobian's eigenvalues for some S. The increment between consecutive S is 5.}
    \label{fig:jacobian-eigenvalues}
\end{figure}

Further we approximate the Lyapunov spectrum as a measure for the divergence of nearby trajectories to obtain information about the system's stability properties.
The Lyapunov spectrum is formally defined as
$$
\lambda_i = \limsup\limits_{t\rightarrow\infty} \frac{\ln{\alpha_i}}{2t}
$$
where $\alpha_i$ are the eigenvalues of $M(t)M^T(t)$. 
Here $M$ denotes the discrete time evolution operator.
We carry out the numerical approximation of the lyapunov spectrum with ChaosTools \cite{Datseris2018}. 
The results are presented in figure \ref{fig:lyapunov-spectrum}.
Lyapunov exponents can be seen as a simple characterization for the stability of manifolds, where a negative exponents indicate attraction, positive exponents indication repulsion and an exponent of zero indicates conservation.

\begin{figure}[h]
    \centering
    \includegraphics[width=0.75\textwidth]{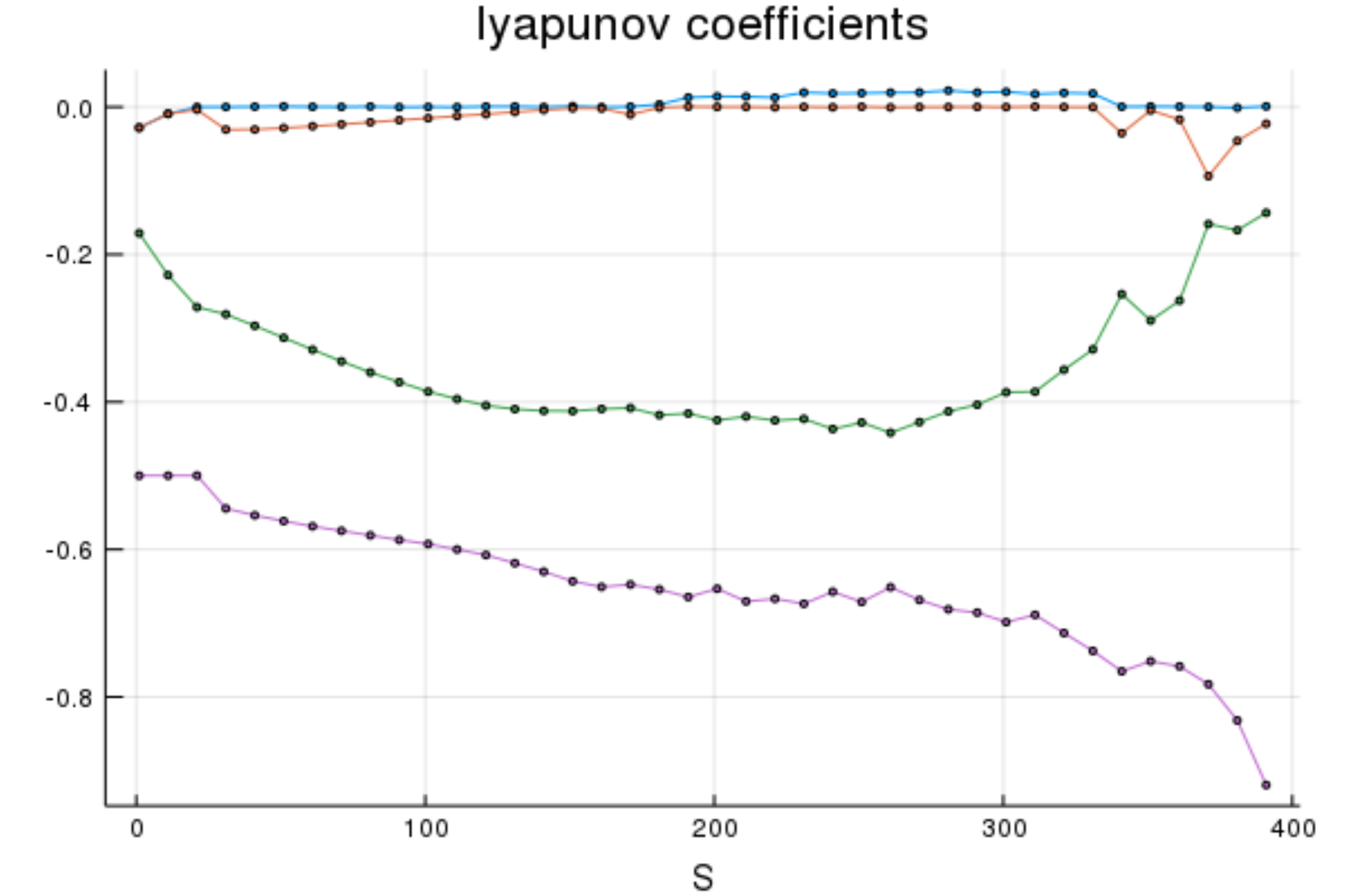}
    \caption{The Lyapunov spectrum of the disease dynamics model for different choices of S.}
    \label{fig:lyapunov-spectrum}
\end{figure}

These figures together suggest a Andronov-Hopf bifurcation around $S\approx20$, where in the interval $[0,~20)$ the fixed point is a stable one.
After this we see a maximal Lyapunov coefficient of value zero paired all other coefficients negative, which is associated with stable cycling.
Around $S\approx180$ we see that the largest Lyapunov coefficient gets positive.
This is possibly associated with the onset of chaos.
Further around $S\approx340$ the systems gains stability again, which is in turn possibly associated with the end of chaotic behavior, returning to stable cycling again.
Around $S\approx100$ we see the two real eigenvalues becoming complex.
We failed to associate this observation with any phenomenon.

Now that we have worked out the coarse system structure we move on to confirm details computationally.
We start by approximating solutions for arbitrary S from each identified interval, namely $[5, 100, 180, 255, 340, 400]$, with algorithm \ref{algo:icn} with tolerance $\varepsilon=10^{-7}$, time step $\Delta t=0.01$ and the maximum number of iterations $I=10$.
The results are visualized in figure (\ref{fig:icn-exploration}).
It can be clearly seen that for $S=5$ the fixed point is attracting, while all other choices of S yield oscillations, which is on par with the previous computational analysis of the Jacobian and Lyapunov spectrum. The choice $S=255$ suggests either an unstable solver or chaotic cycling behavior. Please note also that solver takes up some time to settle, i.e. moving from the initial condition into an orbit.

\begin{figure}[h]
    \centering
    \includegraphics[width=0.3\textwidth]{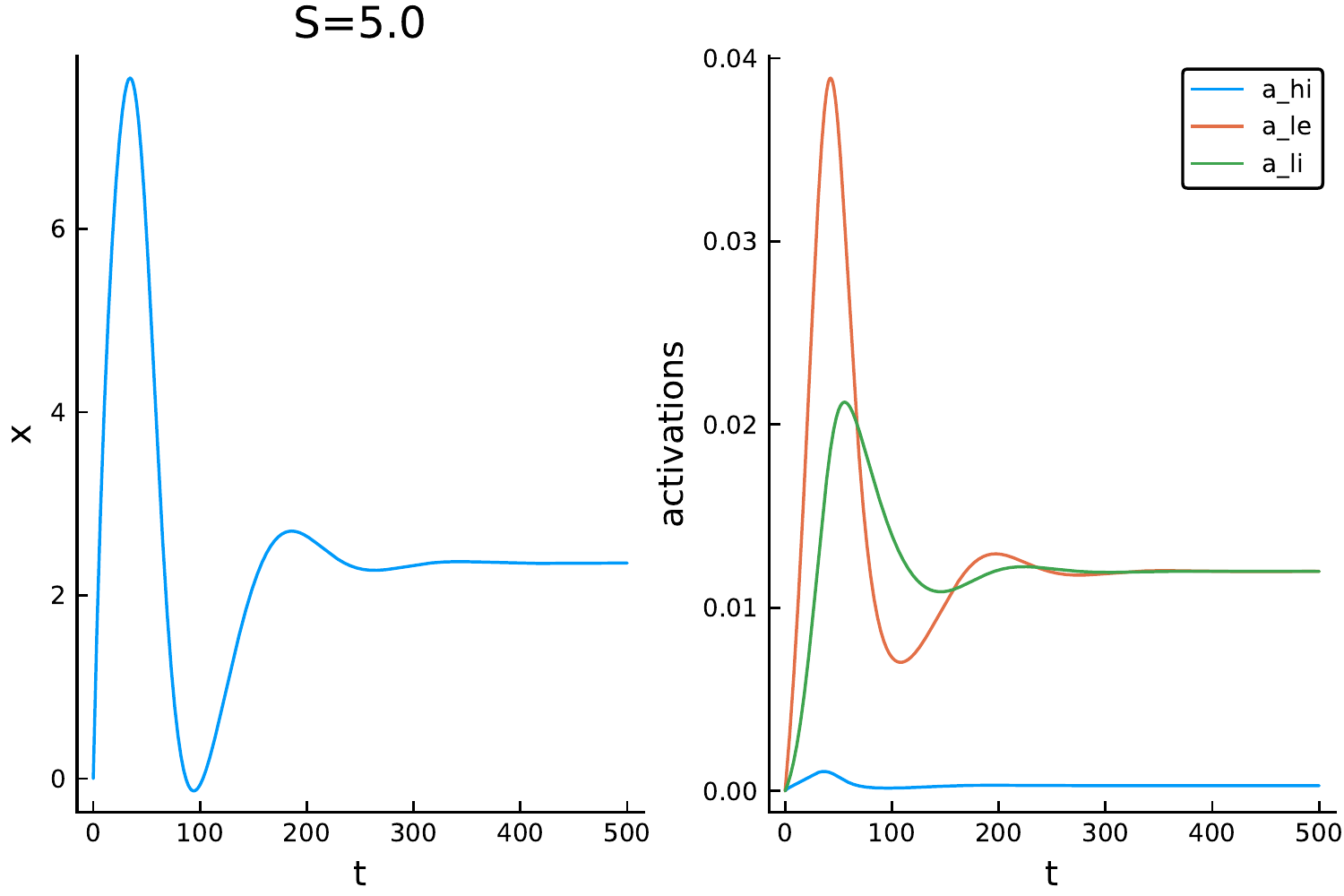}
    \includegraphics[width=0.3\textwidth]{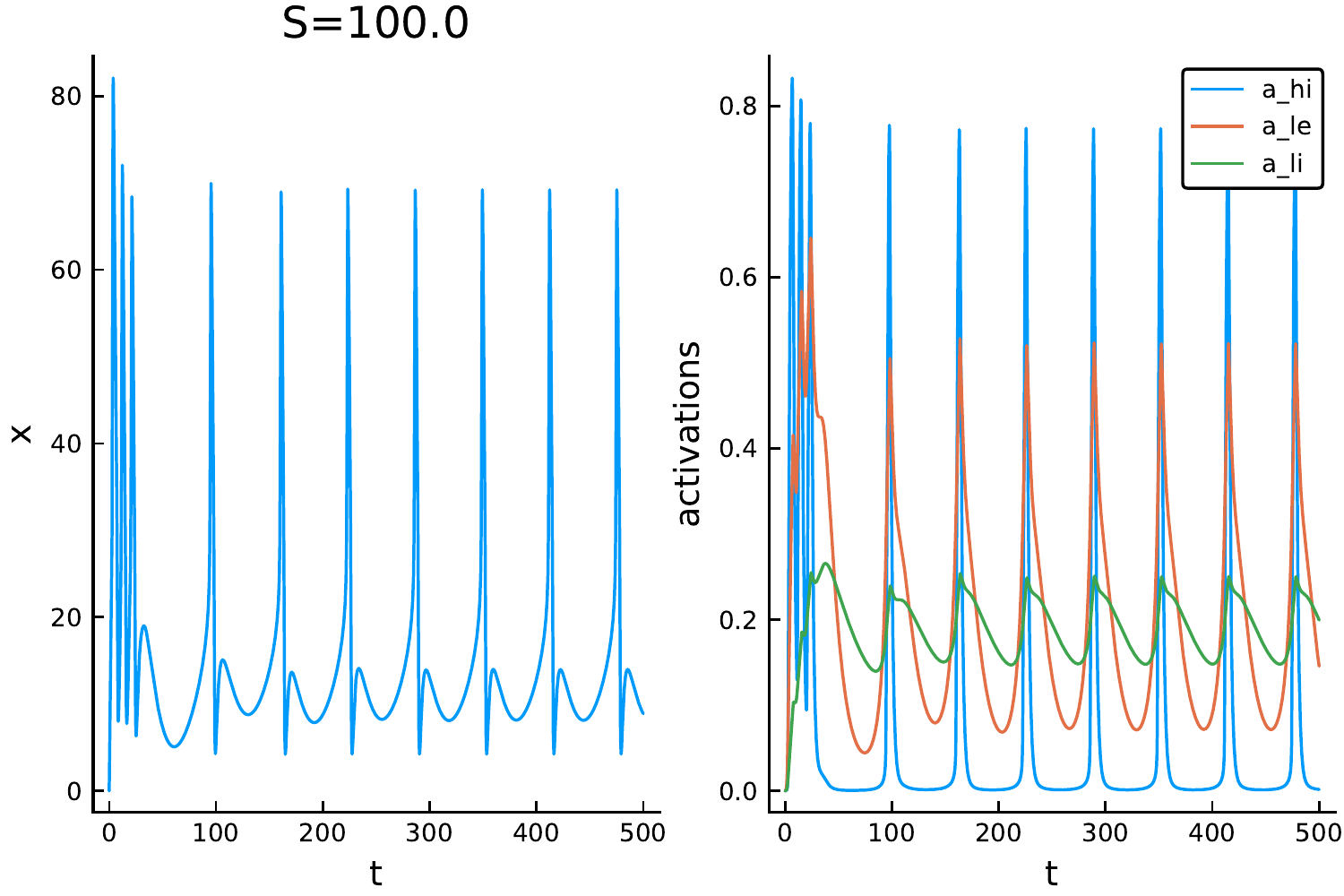}
    \includegraphics[width=0.3\textwidth]{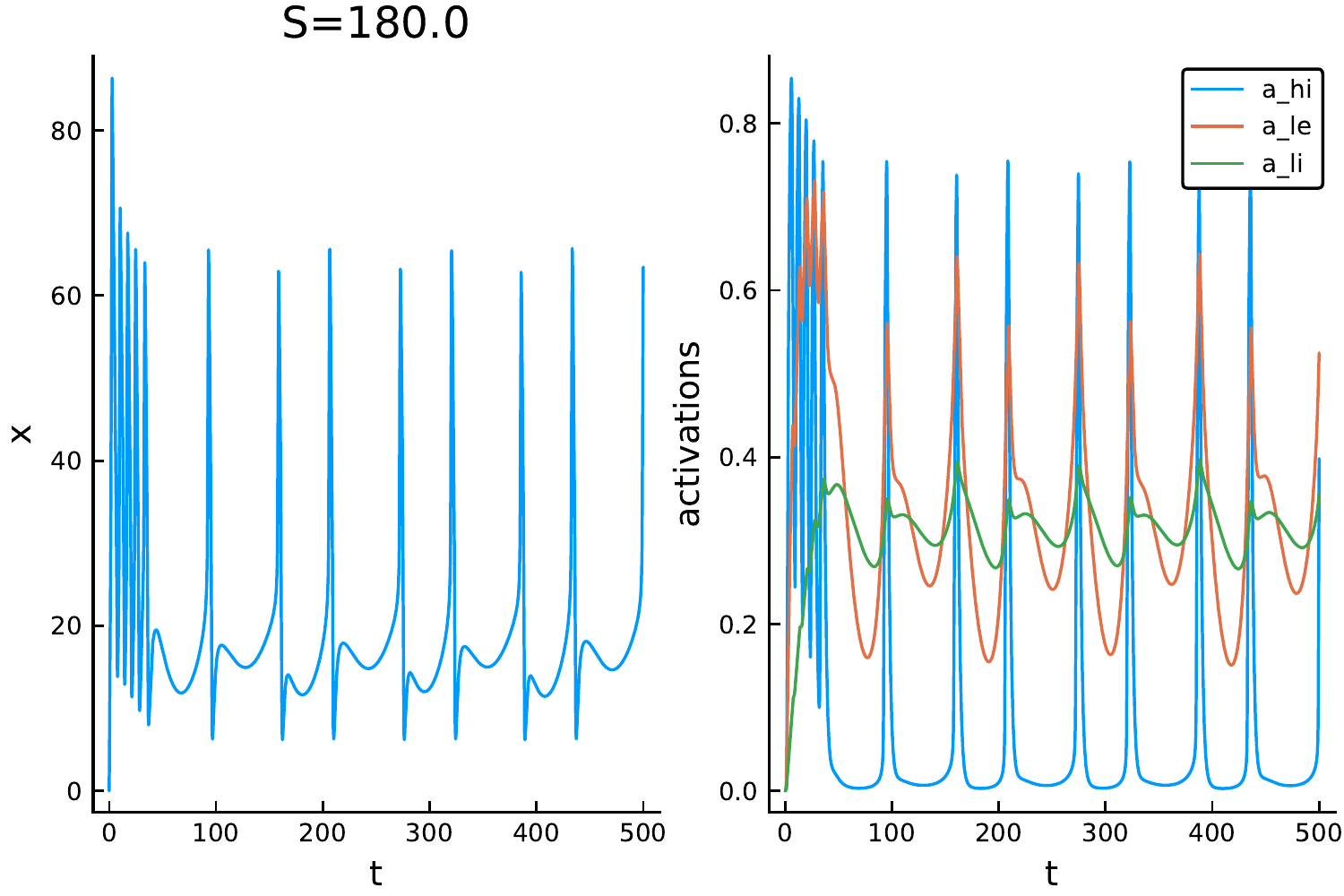}\\
    \includegraphics[width=0.3\textwidth]{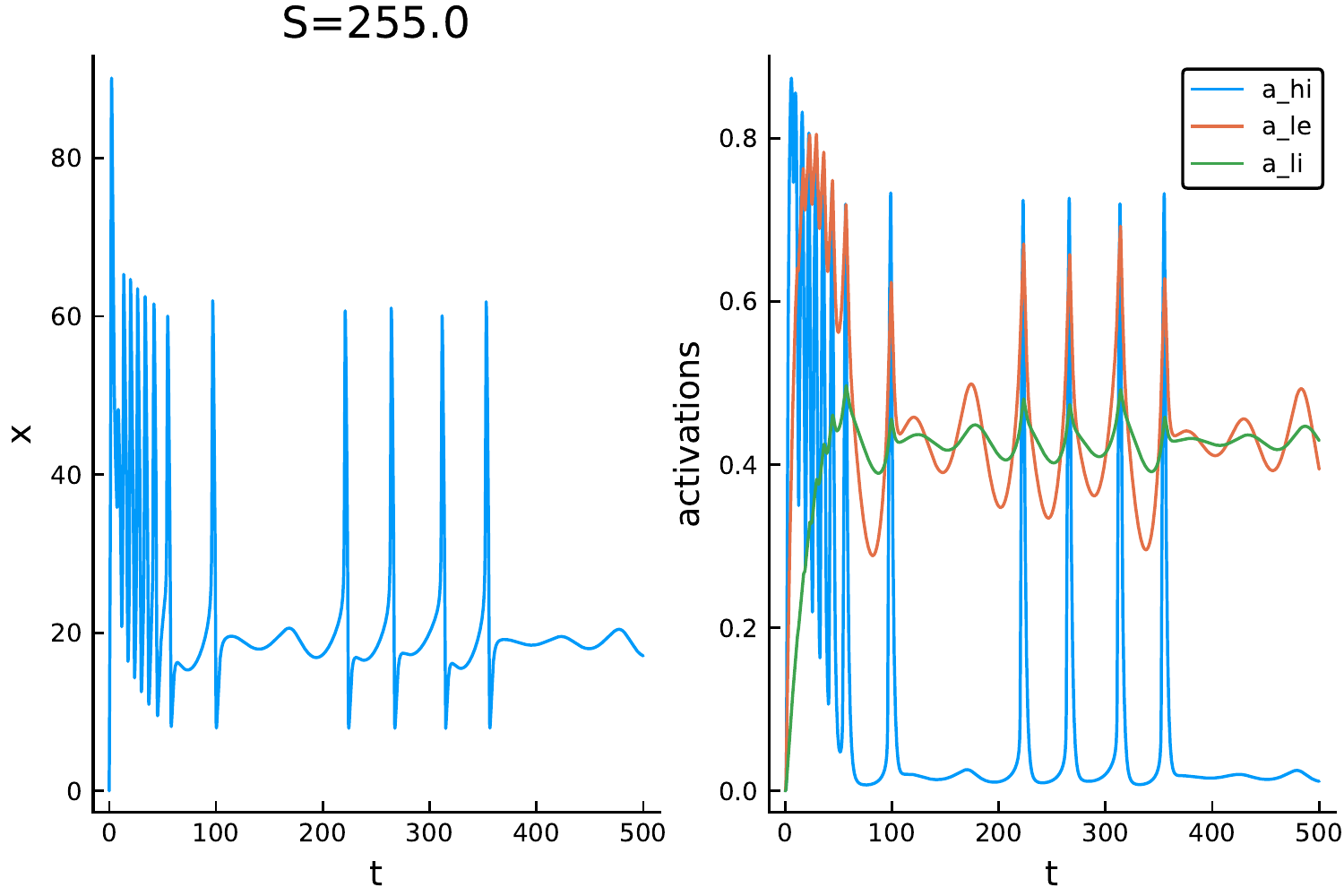}
    \includegraphics[width=0.3\textwidth]{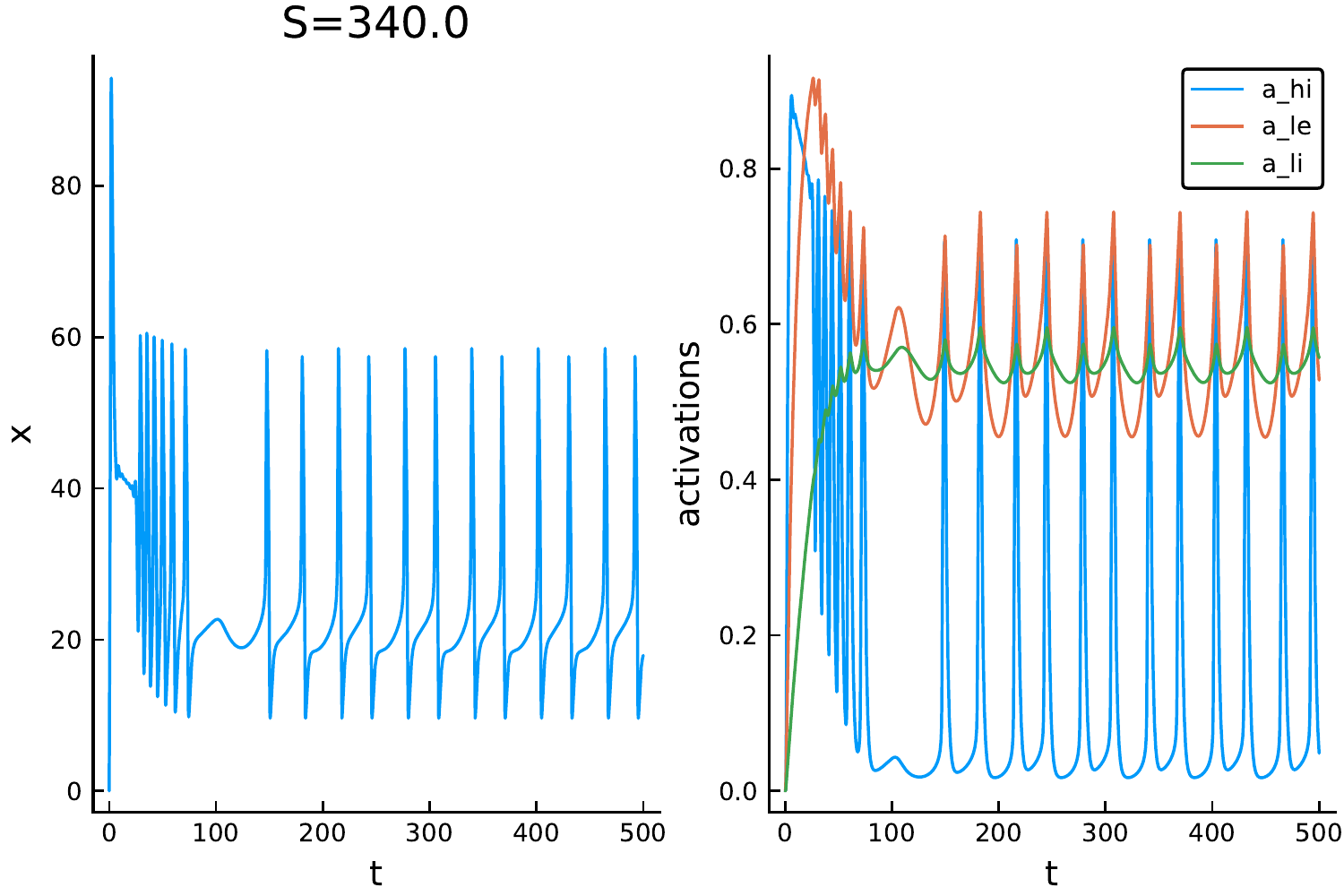}
    \includegraphics[width=0.3\textwidth]{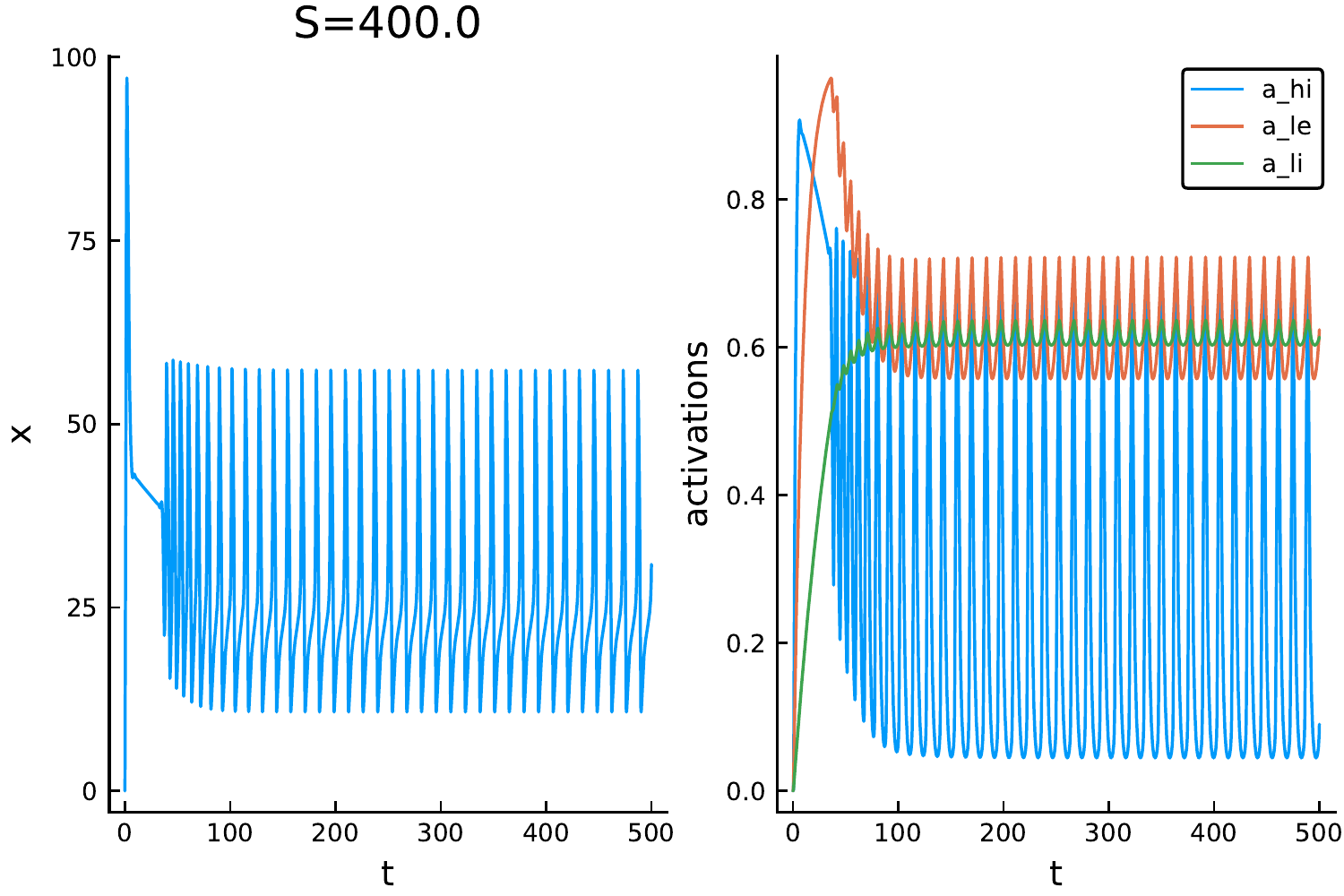}
    \caption{Approximations of the disease dynamics model with the ICN solver (algorithm \ref{algo:icn}) and various S.
    Six approximations for interval $[0,500]$ and initial condition the zero vector, i.e. $u(0) = (0, 0, 0, 0)$, can be seen in pairs of two images, where the left image is the observable $x$ and the left one contains the activation vector ${\bf a}$.
    We have chosen a tolerance $\varepsilon=10^{-7}$, a time step $\Delta t=0.01$ and a maximum number of iterations $I=10$.}
    \label{fig:icn-exploration}
\end{figure}

With this basic structural guesses we move forward towards a computational bifurcation analysis, as to the best of our knowledge no analytic work is available about the general structural properties of Hodgkin-Huxley type systems and especially our disease dynamics model.
We will use two related techniques to quantify the systems behavior computationally, namely Poincaré maps and interspike intervall  (ISI) distributions.
For both techniques we will use the same section.
This will also give us some clues about very basic stability and correctness properties of the in the previous section constructed solvers.

Poincaré sections allow us to study the behaviour of continuous high-dimensional system with a geometric description
in a lower-dimensional space, see \cite{teschl2012}.
The basic idea is to reduce the system to a continuous mapping $T$ of the applied plane $S$ into itself, means we have:
\begin{equation*}
P_{K^+1} = T(P_k) = T[T(P_{k-1)}] = T^2(P_{k-1}) = \ldots
\end{equation*}
Therefore we reduce the continuous flow into a discrete-time mapping.
The Poincaré section of the hyperplane $<(1,0,0,0),u> = 40$ can be seen in figure (\ref{fig:poincare-3d}).

\begin{figure}[h]
    \centering
    \includegraphics[width=0.6\textwidth]{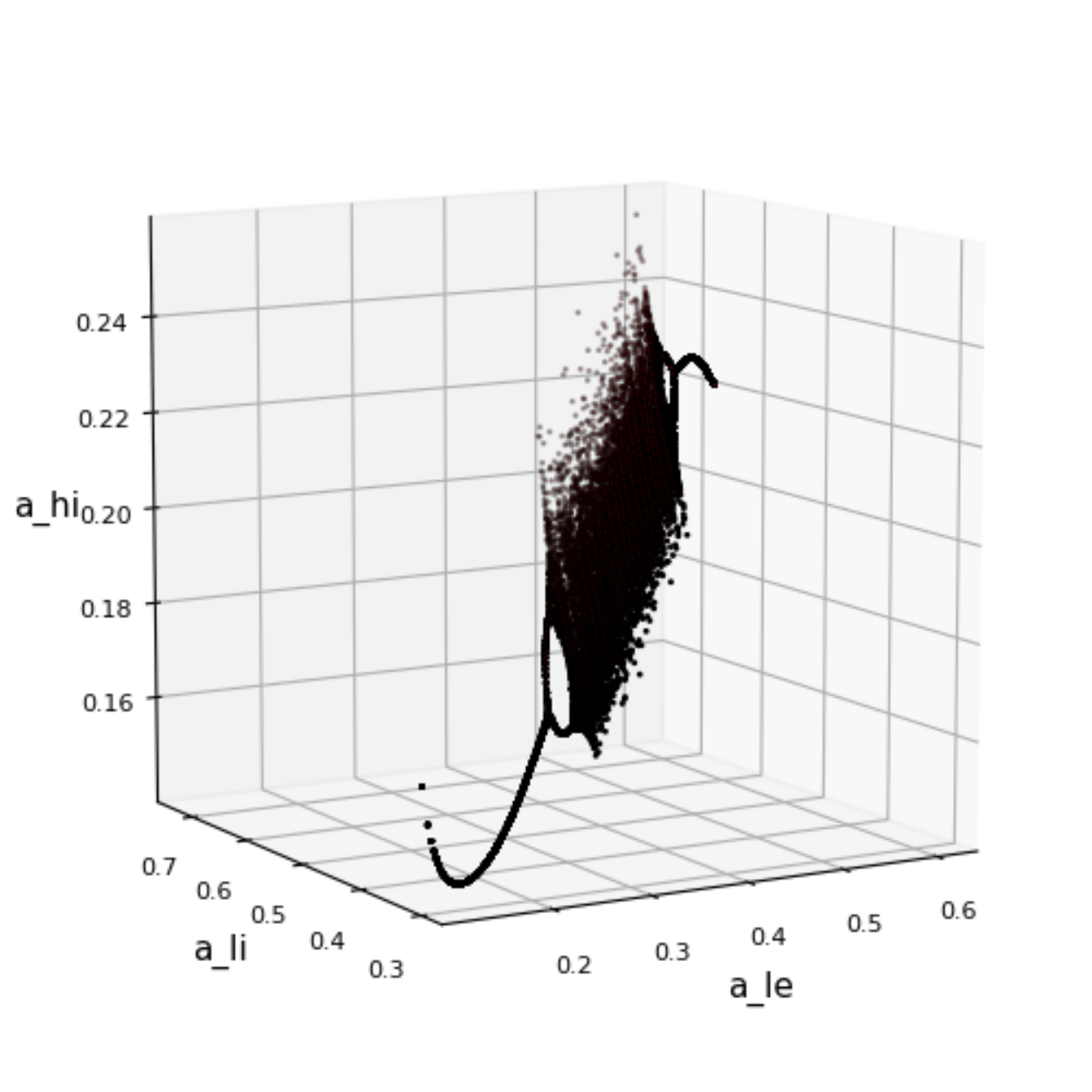}
    \caption{The Poincare section for the disease dynamics with the hyperplane $<(1,0,0,0),{\bf u}> = 40$ with increments of 1 on S over the previously mentioned region of interest $[0,400]$.}
    \label{fig:poincare-3d}
\end{figure}

A closer look into the regions with the first branch and the last merge reveals can be found in figure (\ref{fig:roads-to-chaos}).
The found structures can be identified as classical period doubling and period halving, which are more pointers towards the existence of chaotic behavior, as they usually indicate the onset and the end of chaotic regimes.
Computing the position of three branching points trough a finer step size for S in the Poincaré section, starting with the second branching point (i.e. $\approx (175.1, 178.8, 179.6)$, yields a ratio close to Feigenbaum's constant, suggesting period doubling.
The same structure can be found on the other side at the end of the hypothetically chaotic regime, suggesting period halving (i.e. $\approx (342.25, 340.0, 339.5)$).
While a rigorous analysis is out of the scope of this paper, we take the worked out arguments to support the assumption, that chaotic cycling is actually present as a property of the dynamical system and not as a numerical artifact of instabilities in the used solvers.

\begin{figure}[h]
    \centering
    \includegraphics[width=0.45\textwidth]{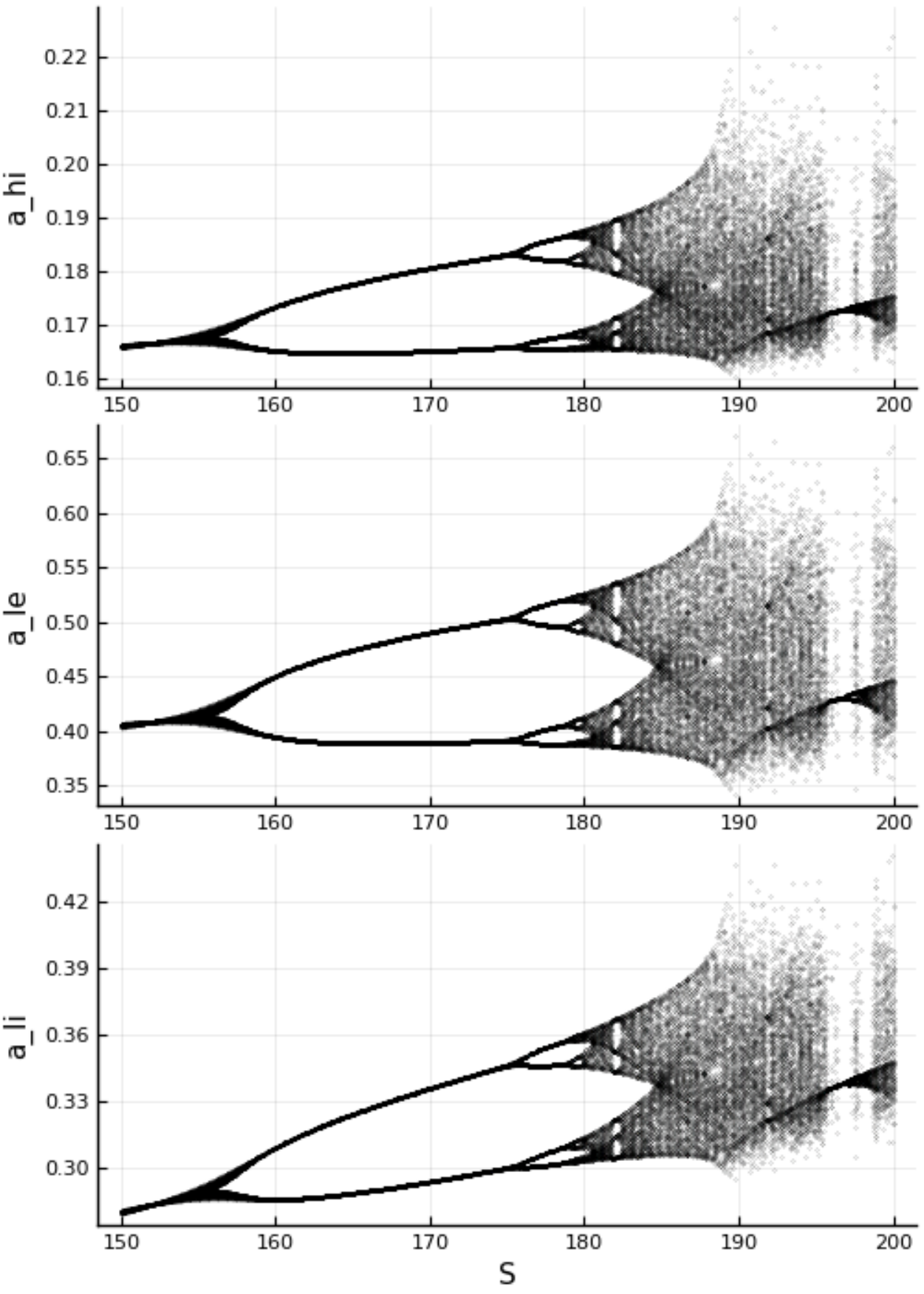}
    \includegraphics[width=0.45\textwidth]{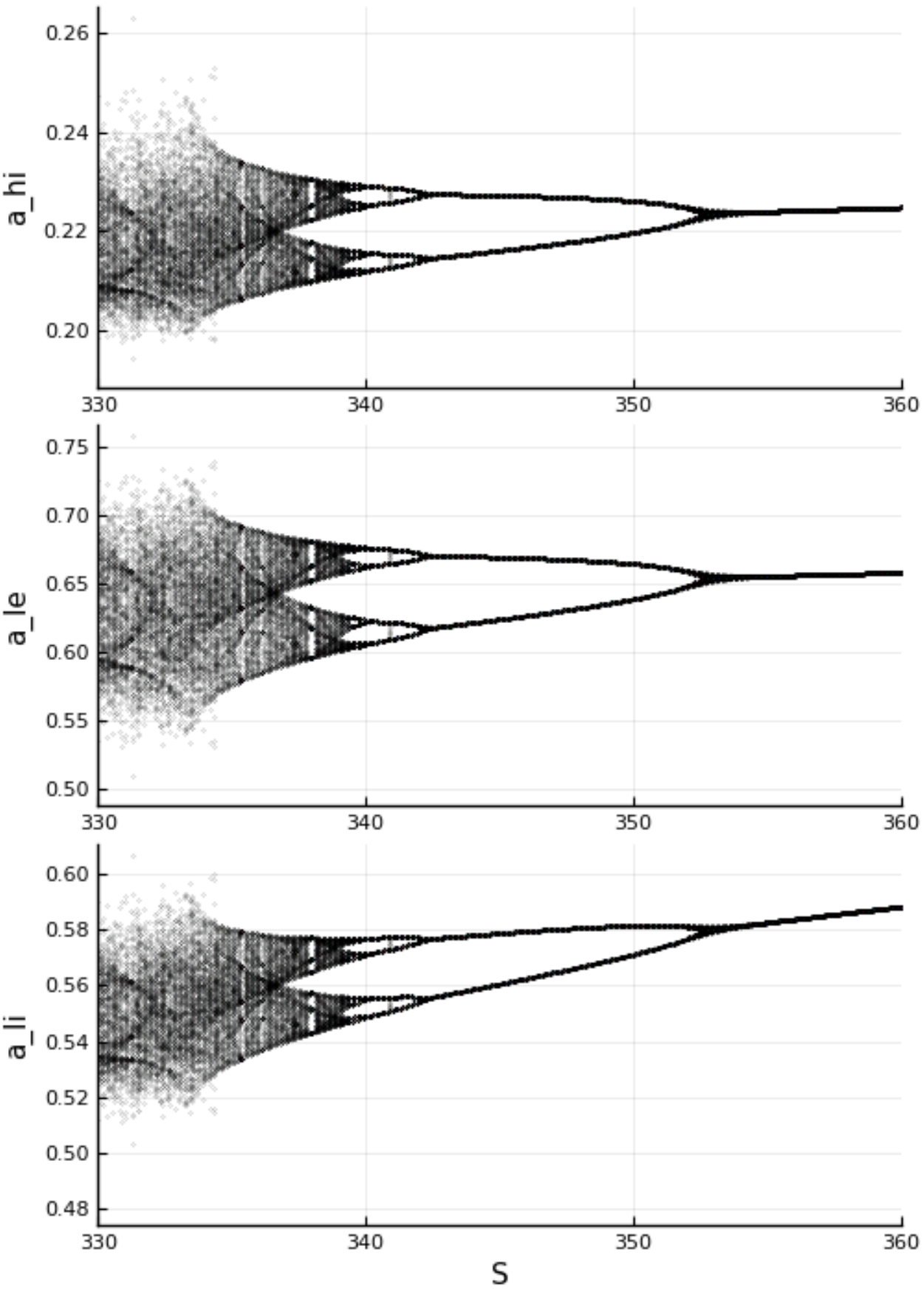}
    \caption{The interspike interval for the disease dynamics with the hyperplane $<(1,0,0,0),u> = 40$ with increments of 1 on S over the previously mentioned region of interest $[0,400]$.}
    \label{fig:roads-to-chaos}
\end{figure}

As a next step we generate the interspike interval distributions for the same section, which is basically the distribution of time between two consecutive intersections of this plane of the solution, which starts in the corresponding attractor.
This distribution is approximated by fixing S and solving the system for a fixed time interval (here [0,10000]).
The bifurcation plot for each in this paper defined scheme can be found in figure (\ref{fig:isi-distributions}).

\begin{figure}[h!]
    \centering
    \fbox{\includegraphics[width=0.3\textwidth]{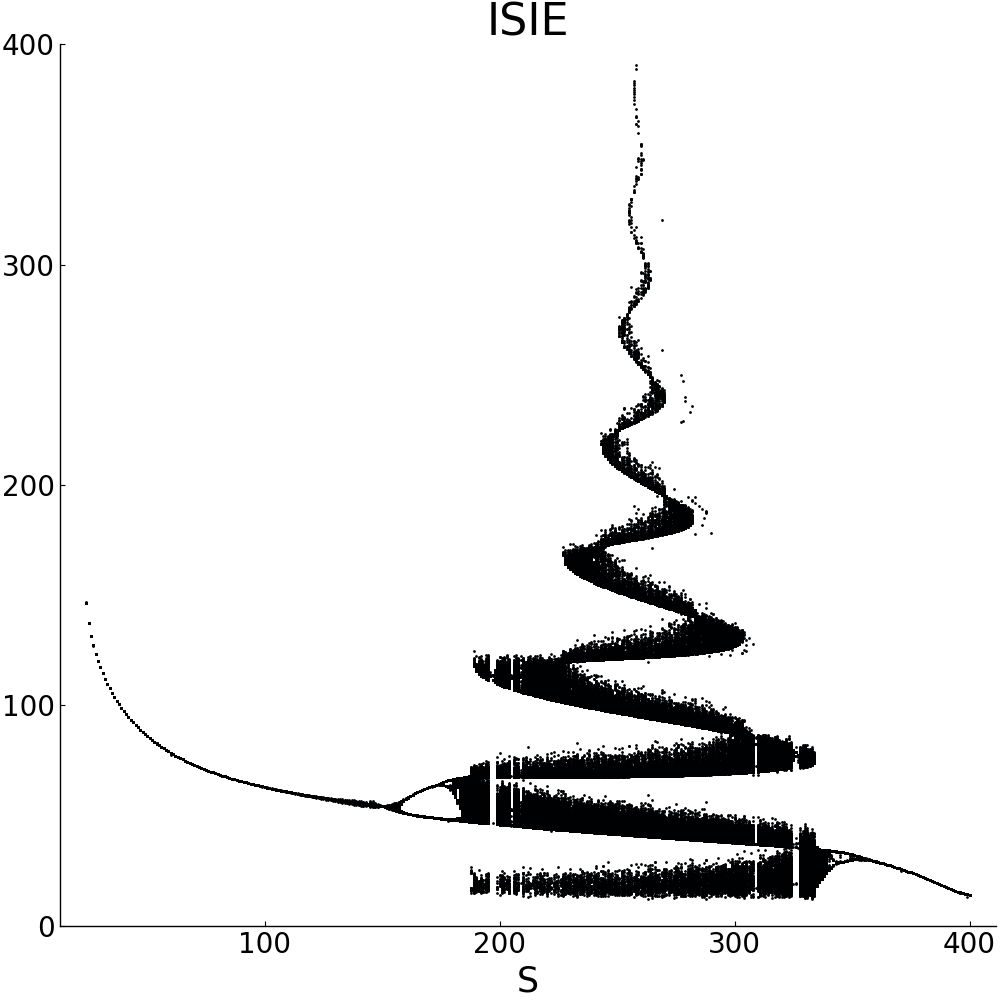}}
    \fbox{\includegraphics[width=0.3\textwidth]{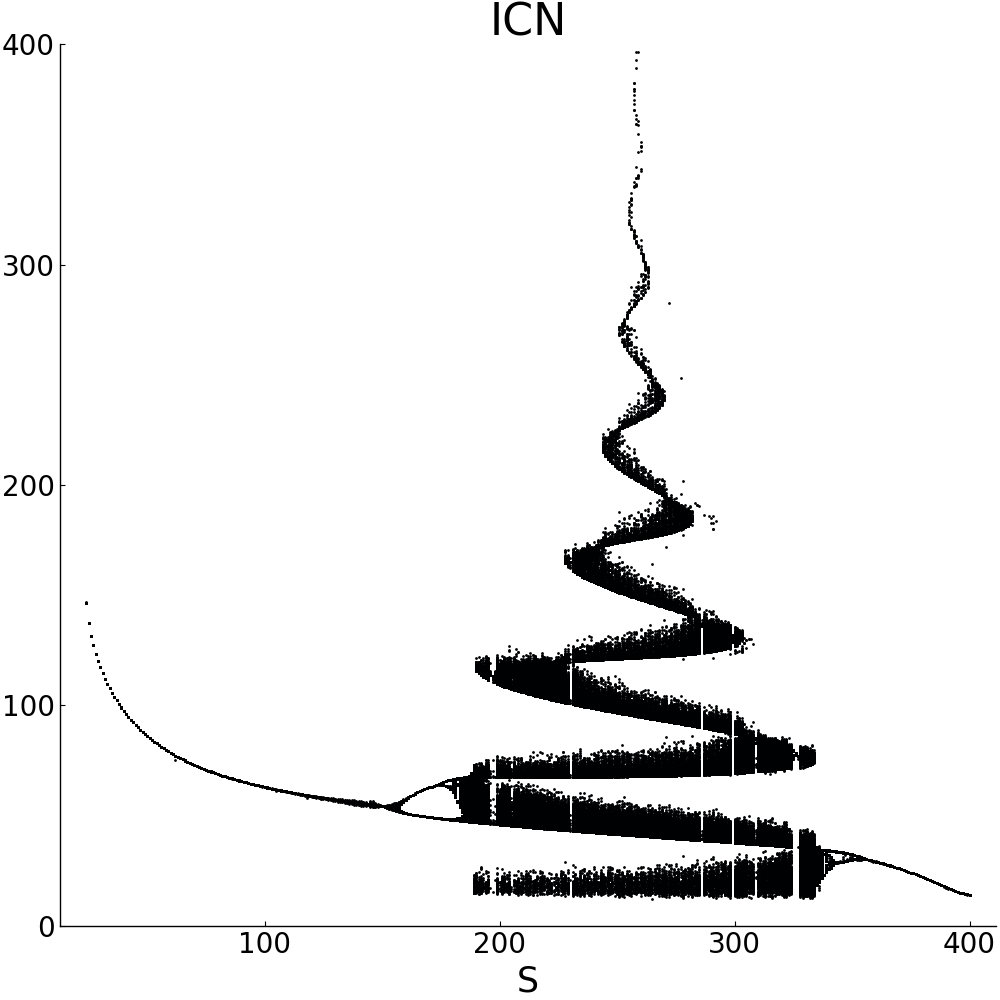}}
    \fbox{\includegraphics[width=0.3\textwidth]{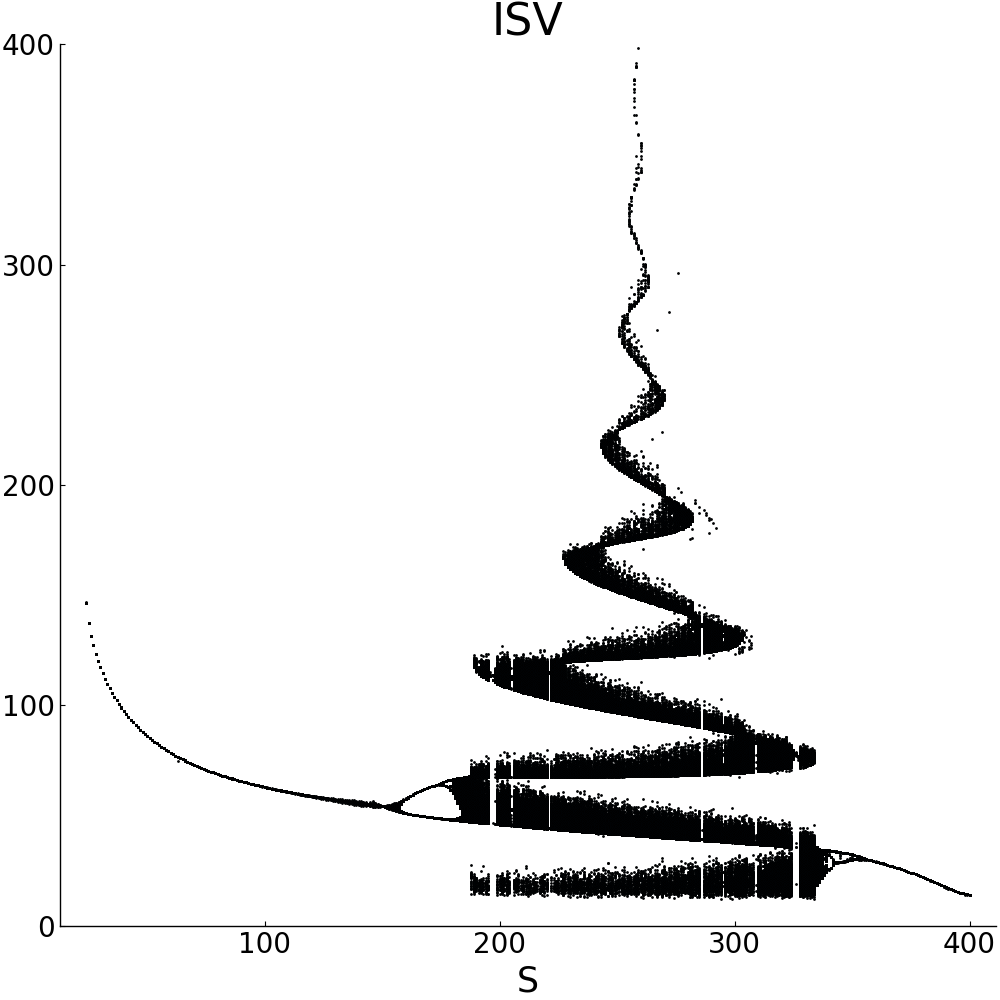}}\\
    \fbox{\includegraphics[width=0.3\textwidth]{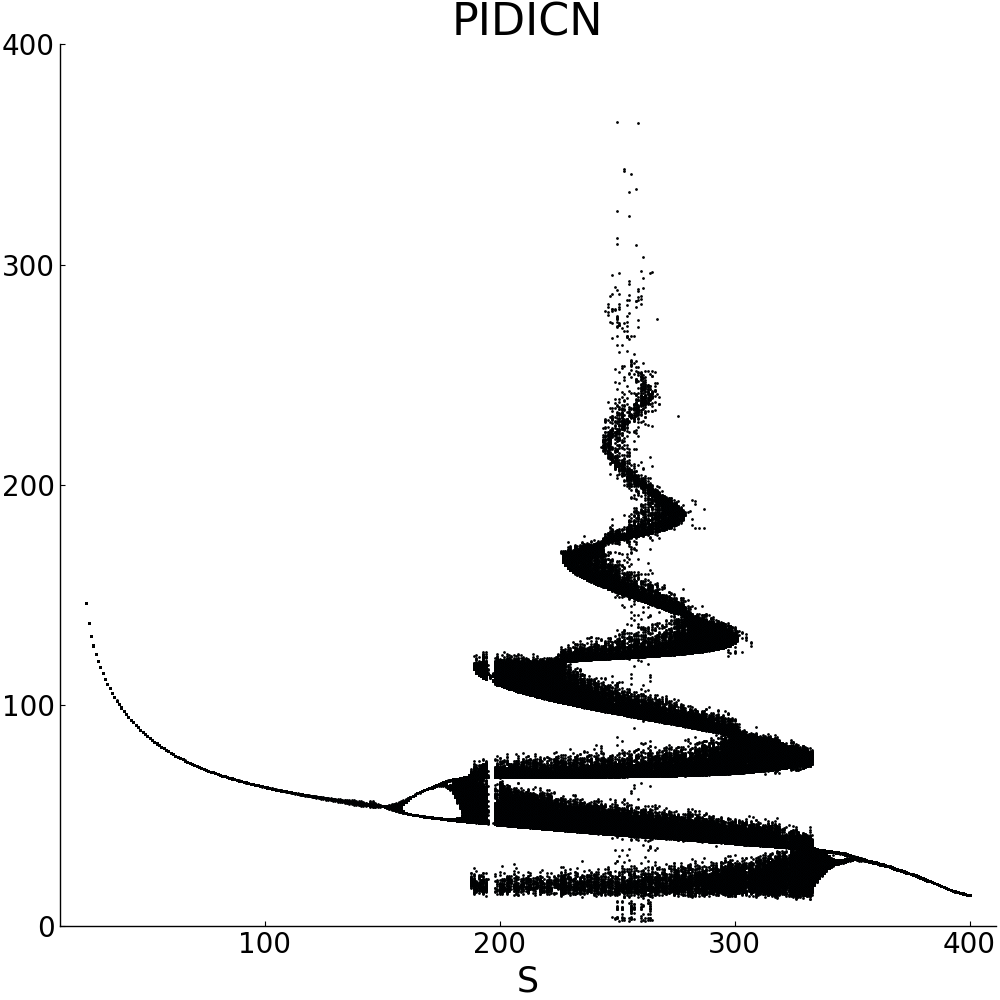}}
    \fbox{\includegraphics[width=0.3\textwidth]{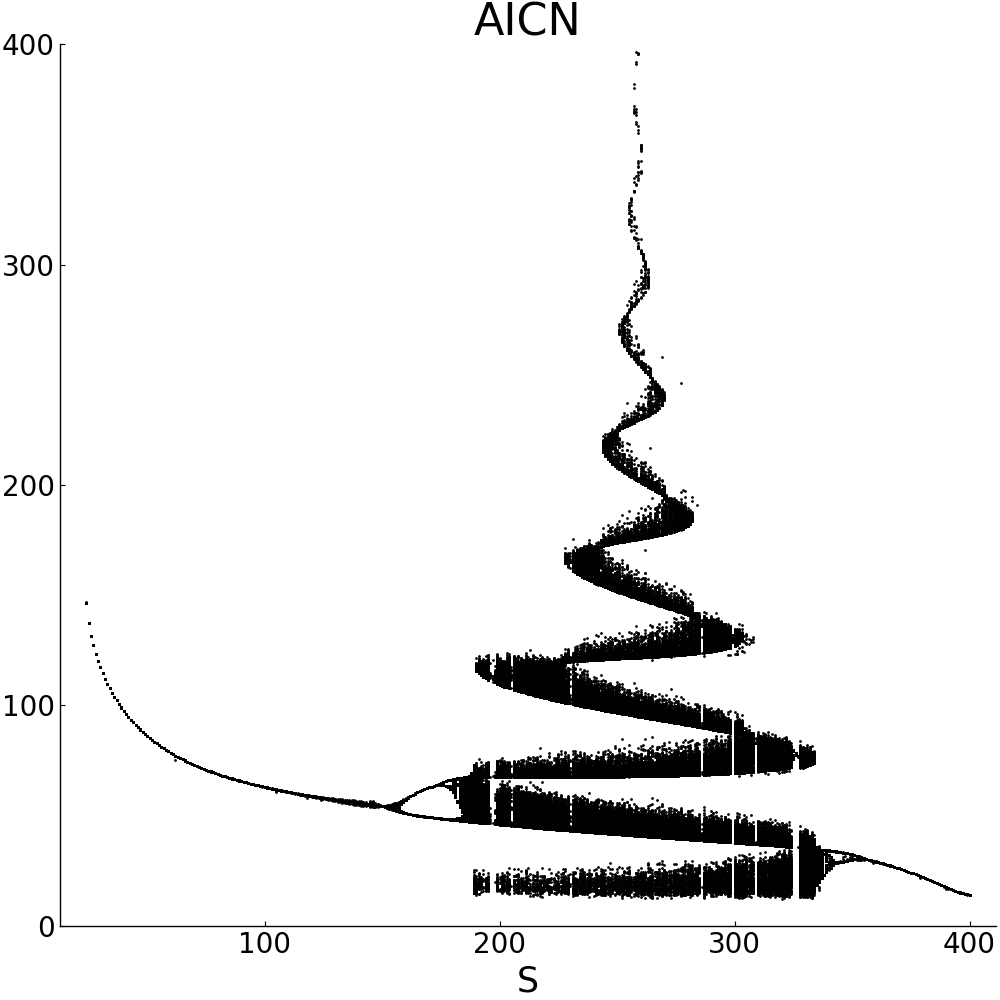}}
    \fbox{\includegraphics[width=0.3\textwidth]{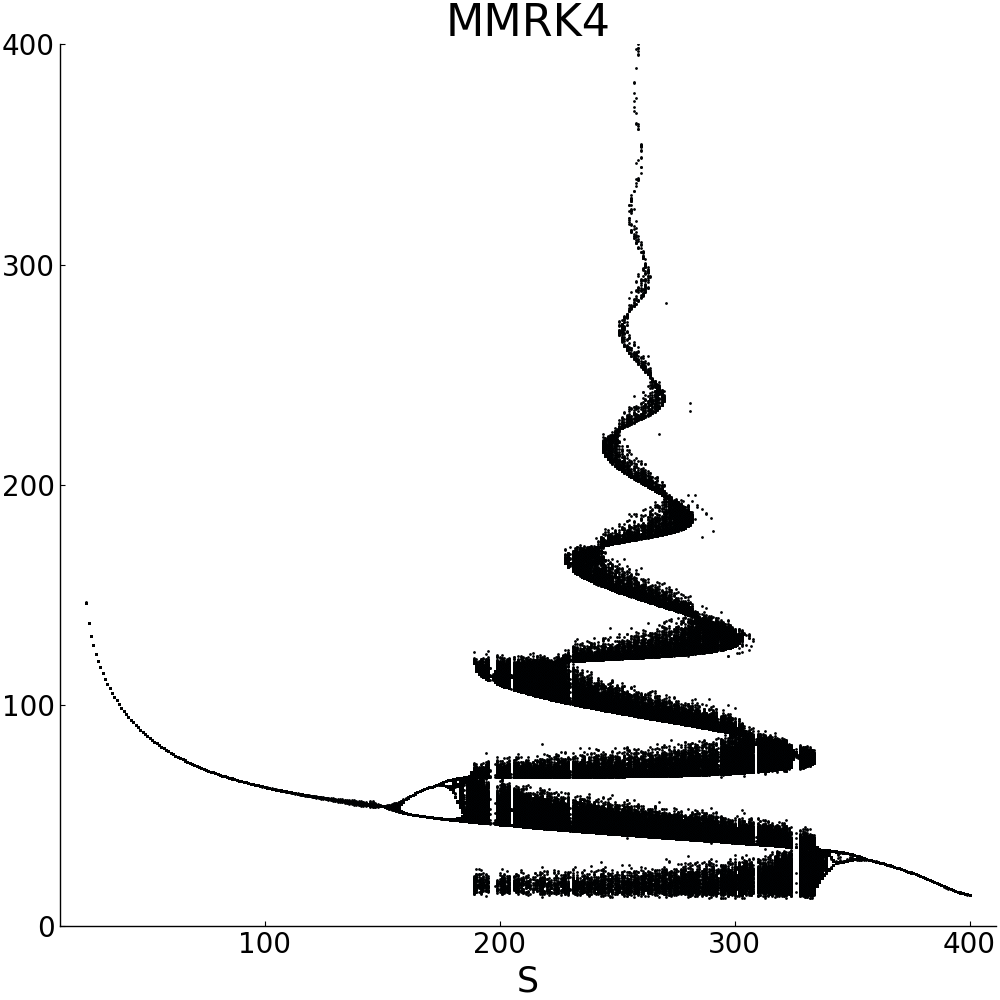}}\\
    \fbox{\includegraphics[width=0.3\textwidth]{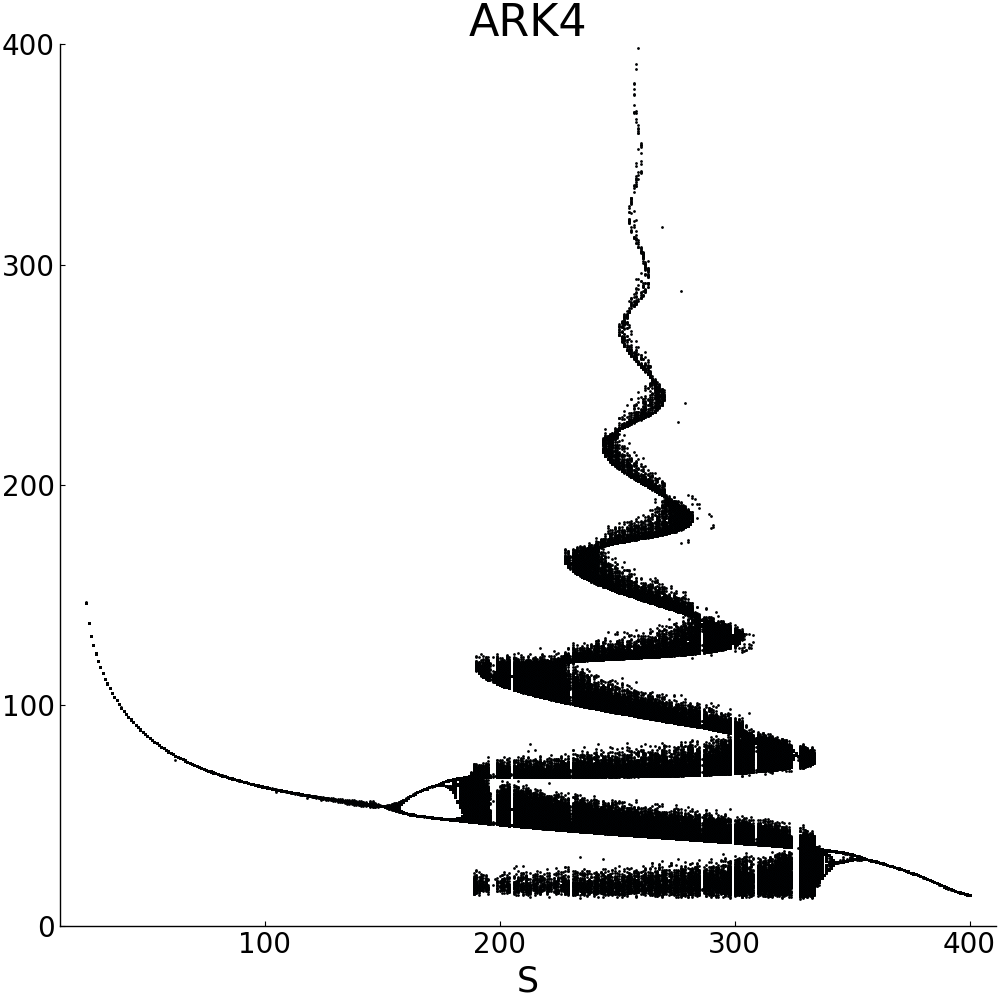}}
    \fbox{\includegraphics[width=0.3\textwidth]{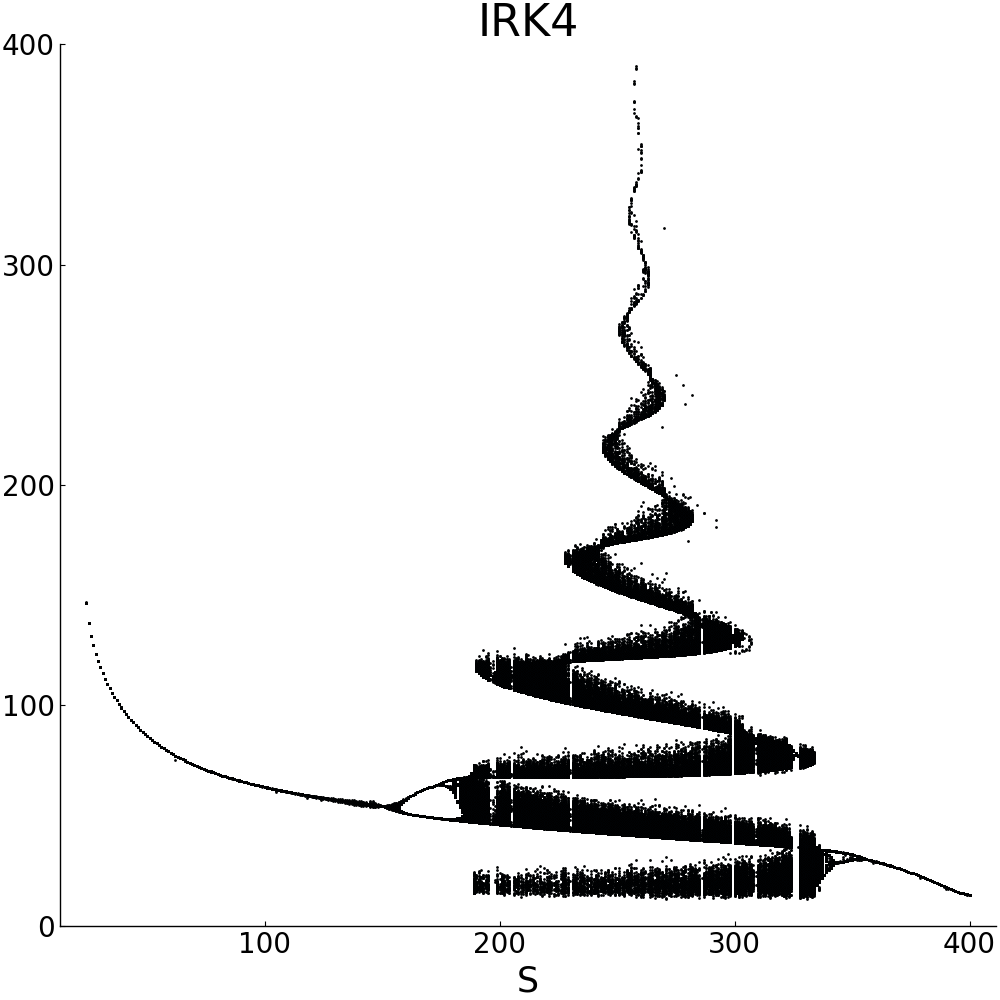}}
    \fbox{\includegraphics[width=0.3\textwidth]{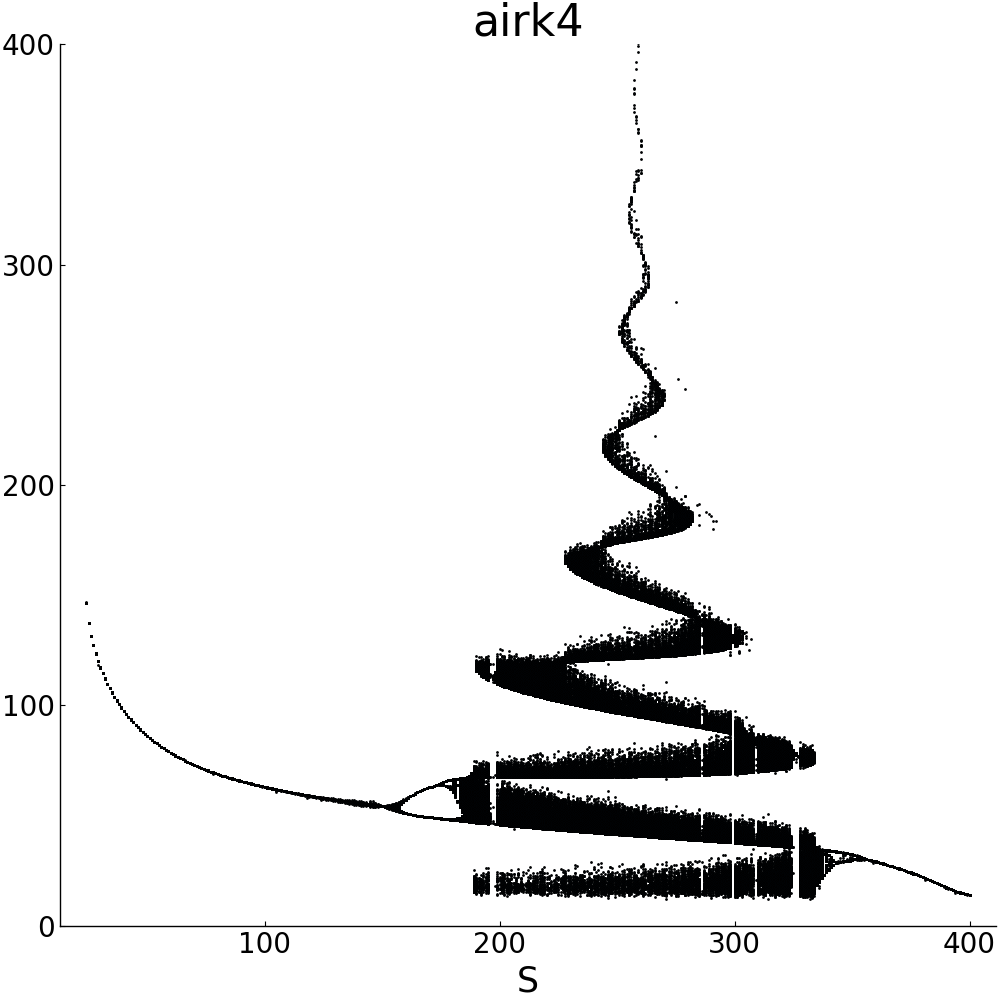}}
    \caption{Interspike interval distributions for each scheme. The x-axis represents our region of interest S from 0 to 400. 
    The upper row shows from left to right the solutions of algorithm \ref{algo:isie}, \ref{algo:icn} and \ref{algo:isv}, while the bottom row shows from left to right \ref{algo:mmrk4} and \ref{algo:irk4}. 
    The first part of the trajectory in $[0,500]$ is ignored analysis to give the solvers a chance to settle properly, allowing to uncover the actual structure of the oscillatory pattern within the attracting region.
    We can see basic agreement on the diagram for all solvers excepting the PIDICN, which seems to smear out the structure in the highly chaotic regime.}
    \label{fig:isi-distributions}
\end{figure}

\subsection{Convergence Study}
Now we test the convergence behaviour of the schemes with fixed time step.
We arbitrarily take one configuration of S for the stable as well as the chaotic cycling, namely $S\in\{100, 253\}$.
The convergence analysis is conducted as follows.
We start the first approximation with initial condition $u(0) = (0,0,0,0)$ and a time step $\Delta t=0.5$. 
With each consecutive approximation we reset the inditial condition and halve the time step while fixing all other parameters. 
The fixed parameters are $\varepsilon=10^{-7}, I=5$.
Consecutive approximations are now compared at the overlapping time points by integrating over the difference of these consecutive approximations.
The results can be found in figure (\ref{fig:convergence-study}).
\begin{figure}[h]
    \centering
    \includegraphics[width=0.48\textwidth]{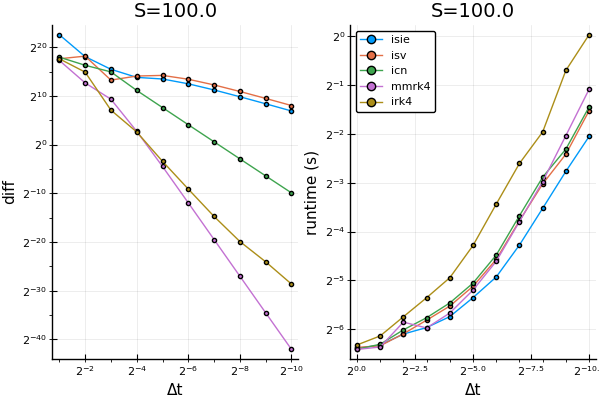}
    \includegraphics[width=0.48\textwidth]{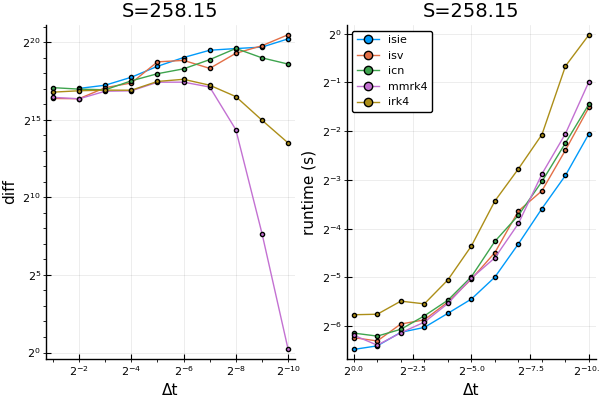}
    \caption{Convergence study for the fixed time step schemes. The parameters have been fixed to $\varepsilon=10^{-7}, I=5$. 
    The left pair of plots is in the regime of stable cycling while the right pair is in the chaotic regime.
    Each pair shows the integral error between two consecutive time step halvings and the corresponding runtime for each scheme.
    Note that the plots are on a log-log scale as we want to highlight the correlation on halving the time step consecutively.
    This way of plotting directly reveals the order of convergence, which correlates in the stable cycling case with the curve's slope.
    The computations were carried out on an Intel Core i5-7200U.}
    \label{fig:convergence-study}
\end{figure}

If the error shrinks with each halving the scheme converges to a solution, which should in the case of stable cycling be the corresponding solution of our system.
In the case of chaotic cycling the solution converges only for this specifically given time interval, as in the presence of chaos nearby trajectories diverge with exponential speed.
This divergence cannot be handled in general by our solvers for long time scales. 
On the one hand we usually cannot hit a solution exactly with only machine precision available, which may already be another solution trajectory which diverges exponentially.
On the other hand we can, again by machine precision limited, not reduce the time step for an arbitrarily large time interval, as computations get unstable for too small time steps.

\subsection{Analysis of the Adaptive Schemes}
Finally we want to evaluate the performance of the adaptive schemes, namely algorithms \ref{algo:pidicn}, \ref{algo:aicn}, \ref{algo:ark4} and \ref{algo:airk4}, side by side with the fixed time step schemes.
For this we start by evaluating the following local norms:
\begin{itemize}
\item $L_2$-norm of the solutions for each time-points:
\begin{align*}
\norm{\bf u}_{L_2[t^n, t^{n+1}]} = \sqrt{\Delta t^n \left( x(t^n)^2 + a_{he}(t^n)^2 + a_{li}(t^n)^2 +  a_{le}(t^n)^2 \right)}
\end{align*}
\item $L_2$-norm of the derivations for each time-points:
\begin{align*}
 \norm{\frac{d {\bf u}}{d t}}_{L_2[t^n, t^{n+1}]} 
 =  \sqrt{\Delta t^n \left(\left(\frac{d x(t^n)}{d t}\right)^2 + \sum_{i \in \{hi,le,li\}} \left(\frac{d a_{i}(t^n)}{dt}\right)^2 \right)}
\end{align*}
\end{itemize}
which result in the following global norms:
\begin{itemize}
\item $L_2$-norm of the solutions in the time domain:
\begin{align*}
 \norm{{\bf u}}_{L_2[0, T]}
= \sqrt{ \sum_{n=1}^N \Delta t^n \left( x(t^n)^2 + a_{he}(t^n)^2 + a_{li}(t^n)^2 +  a_{le}(t^n)^2 \right) }
\end{align*}
\item $L_2$-norm of the derivations for each time-points:
\begin{align*}
\norm{\frac{d {\bf u}}{d t}}_{L_2[0, T]} = 
\sqrt{\sum_{n=1}^N \Delta t^n \left( \left(\frac{d x(t^n)}{d t}\right)^2 + \sum_{i \in \{hi,le,li\}} \left(\frac{d a_{i}(t^n)}{dt}\right)^2 \right)}
\end{align*}
\end{itemize}
where the derivations are approximated as $\frac{d x(t^n)}{d t} \approx \frac{x(t^{n+1})- x(t^{n})}{\Delta t}$.
Again we chose two S arbitrarily (here \{100, 258.15\}) from the stable and chaotic cycling regions.
The solvers are configured as follows:
$$
I=10, J=10, \varepsilon_{fp} = 10^{-7}, \varepsilon_{t} = 10^{-7},
\Delta t = \Delta t_0 = 0.01
$$
For the PID-controller we hand-tuned the parameters to the following values: 
$$K_P = 0.025, \; K_I = 0.075, \; K_D = 0.01$$.
The time domain is $[0,10000]$.
The local norms can be found in figure (\ref{fig:local-norm-analysis}) while the global norms are listed in table (\ref{tab:global-norm-analysis}) side by side with benchmarked runtimes.
For the stable cycling we observe that all solvers nearly agree on the given time interval.
Only the ISIE (\ref{algo:isie}) and ISV (\ref{algo:isv}) schemes are a bit off.
No agreement is found in the chaotic case.

\begin{figure}[h]
    \centering
    \includegraphics[width=0.48\textwidth]{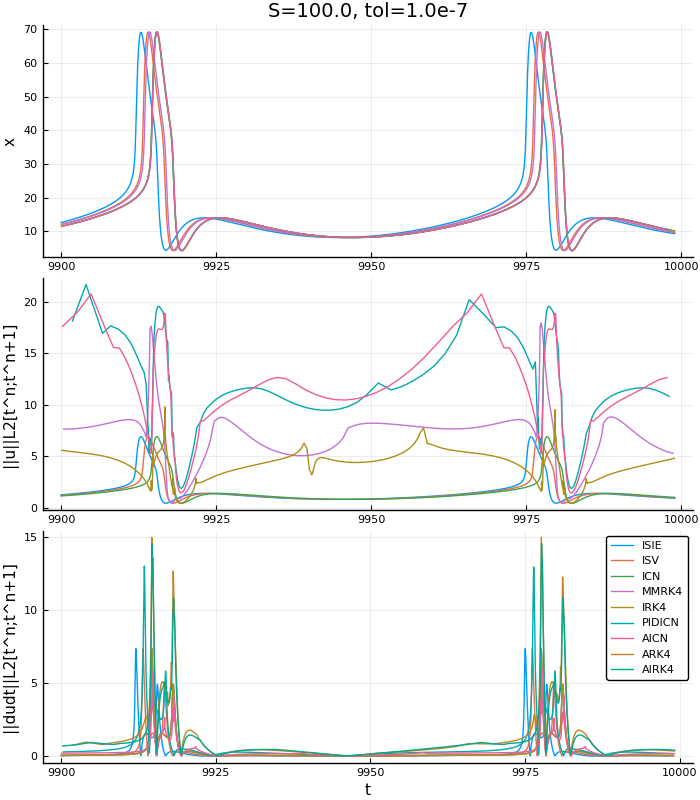}
    \includegraphics[width=0.48\textwidth]{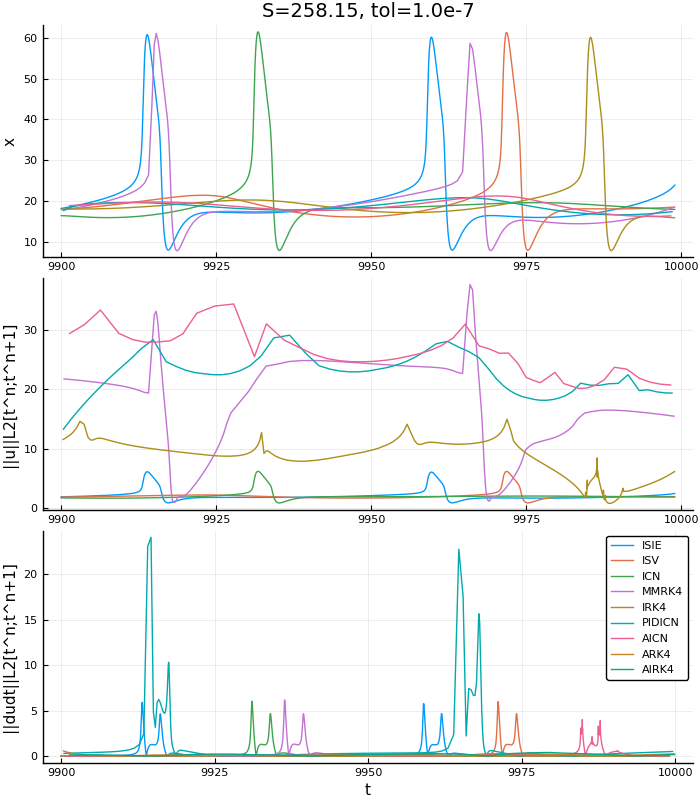}
    \caption{Last 100ms of the local error norms for all solvers.
    The left column shows the previously defined local norms for stable cycling on the example of $S=100$ and the right one the ones for chaotic cycling with $S258.15$.}
    \label{fig:local-norm-analysis}
\end{figure}

\begin{table}[h]
    \centering
    \begin{tabular}{|c||c|c|c|c|c|c|}
    \hline
        & \multicolumn{3}{c|}{S=100} & \multicolumn{3}{c|}{S=258.15} \\ \hline
        \thead{Scheme} & \thead{min time} & \thead{mean time} & \thead{max time} & \thead{min time} & \thead{mean time} & \thead{max time} \\ \hline\hline
        ISIE & 296.279 ms & 299.401 ms & 307.805 ms & 263.207 ms & 266.759 ms & 273.071 ms \\
        ICN & 467.571 ms & 473.274 ms & 479.486 ms & 405.481 ms & 415.413 ms & 453.097 ms\\
        ISV & 428.761 ms & 435.463 ms & 450.706 ms & 393.316 ms & 396.468 ms & 401.683 ms\\
        MMRK4 & 434.206 ms & 451.805 ms & 472.120 ms & 403.617 ms & 414.823 ms & 429.899 ms\\
        IRK4 & 1.149 s & 1.153 s & 1.160 s & 972.086 ms & 1.065 s & 1.135 s\\
        PIDICN & 365.669 ms & 371.045 ms & 382.652 ms & 128.728 ms & 144.721 ms & 203.532 ms\\
        AICN & 929.919 ms & 937.354 ms & 954.432 ms & 689.332 ms & 743.694 ms & 823.602 ms\\
        ARK4 & 61.272 ms s & 64.606 ms & 70.739 ms & 51.398 ms & 60.307 ms & 89.615 ms \\
        AIRK4  & 392.592 ms & 398.475 ms & 413.562 ms & 373.588 ms & 397.572 ms & 434.646 ms\\ \hline
    \end{tabular}
    \caption{Runtimes of the different algorithms for the interval $[0,10000]$.
     The computations were carried out on an Intel Core i5-7200U.}
    \label{tab:benchmark}
\end{table}

\begin{table}[h]
    \centering
    \begin{tabular}{|c||c|c|c|c|}
    \hline
        & \multicolumn{2}{c|}{S=100} & \multicolumn{2}{c|}{S=258.15} \\ \hline
        \thead{Scheme} & $\norm{{\bf u}}_{L_2[0, 9999]} $ & $\norm{\frac{d {\bf u}}{d t}}_{L_2[0, 9999]}$ & $\norm{{\bf u}}_{L_2[0, 9999]} $ & $\norm{\frac{d {\bf u}}{d t}}_{L_2[0, 9999]}$ \\ \hline\hline
        ISIE & 1779.1042482178611 & 763.6817995346803 & 2074.99633068202 & 551.5939982581531   \\
        ICN & 1779.6596093842356 & 765.6633264308293 & 2080.081103110141 & 561.9538618698593  \\
        ISV & 1779.8227558934245 & 764.772539199622 & 2072.520439246565 & 548.0642069057208  \\
        MMRK4 & 1779.4729101969149 & 764.9818721864283 & 2066.8111537452933 & 538.8410380757709 \\
        IRK4 & 1779.659878381908 & 765.6883397814581 & 2094.416156306047 & 584.595355425552 \\
        PIDICN & 1778.7230770630867 & 765.3645009066148 & 2086.96979063115 & 620.090939713537 \\
        AICN & 1779.6596093842356 & 765.6633264308293 & 2080.081103110141 & 561.9538618698593 \\
        ARK4 & 1772.6310255573324 & 764.9749934137184 & 2078.7562825863693 & 565.8865895392348 \\
        AIRK4 & 1773.8147243055737 & 765.2310905681921 & 2068.84774687211 & 546.2333427143521 \\ \hline
    \end{tabular}
    \caption{A tabular view of the previously defined global norms for all solvers. 
    The left side of the table contains the stable cycling case $S=100$ while the right side contains the chaotic cycling case $S=258.15$.}
    \label{tab:global-norm-analysis}
\end{table}

We further capture statistical features of the adaptive schemes fluctuations by computing the expectation and variance as follows, assuming that the same point does not lie on the approximation twice for our chosen time intervals:
\begin{itemize}
\item Expectation:
$$
\mathbb{E}\left[\norm{{\bf u}}_{L_2[t^{n_1}, t^{n_2}]}\right] = \frac{1}{n_2 - n_1}  \sum_{n = n_1}^{n_2}  \norm{{\bf u}}_{L_2[t^n, t^{n+1}]}
$$
\item Variance:
$$
\mathbb{V}\left[\norm{{\bf u}}_{L_2[t^{n_1}, t^{n_2}]}\right] = \frac{1}{n_2 - n_1}  \sum_{n = n_1}^{n_2} \left( \norm{ {\bf u}}_{L_2[t^n, t^{n+1}]}  - \mathbb{E}\left[\norm{{\bf u}}_{L_2[t^{n_1}, t^{n_2}]}\right] \right)^2
$$
\end{itemize}
where $t^{n_1} < t^{n_2}$.
The results can be found in table (\ref{tab:statistical-features}).
We can take from these tables that the adaptive RK4 schemes (\ref{algo:ark4} \& \ref{algo:airk4}) can handle larger time steps while bounding the local error.
Note that this does not help in the assumed chaotic case, as in chaos nearby trajectories diverge with exponential speed.
Still, with a small enough error bound we are able to somewhat bound the global error for small time intervals.

\begin{table}[h]
    \centering
    \begin{tabular}{|c||c|c|c|c|}
    \hline
        & \multicolumn{2}{c|}{S=100} & \multicolumn{2}{c|}{S=258.15} \\ \hline
        \thead{Scheme} & $\mathbb{E}\left[\norm{{\bf u}}_{L_2[0, 9999]}]\right]$ & $\mathbb{V}\left[\norm{{\bf u}}_{L_2[0, 9999]}\right]$ & $\mathbb{E}\left[\norm{{\bf u}}_{L_2[0, 9999]}\right]$ & $\mathbb{V}\left[\norm{{\bf u}}_{L_2[0, 9999]}\right]$ \\ \hline\hline
        PIDICN & 1.53167 & 3.94941 & 1.56736 & 8.5948\\
        AICN & 1.10589 & 1.3043 & 0.934257 & 0.936426 \\
        ARK4 & 3.64771 & 14.896 & 2.74606 & 11.2035  \\
        AIRK4 & 3.33507 & 12.3004 & 2.60268 & 9.33504
 \\ \hline
    \end{tabular}
    \caption{A tabular view of the statistical features for the adaptive solvers. 
    The left side of the table contains the stable cycling case $S=100$ while the right side contains the chaotic cycling case $S=258.15$.}
    \label{tab:statistical-features}
\end{table}

\section{Conclusion}
\label{concl}

We started by giving a definition for Hodgkin-Huxley type systems and some characteristics.
Informally we refer to Hodgkin-Huxley type systems as differential equations of a special class of potentially oscillating reaction-diffusion type systems with local activation and inactivation mechanisms.
Based on different assumptions about these systems we derived some solvers and took an characteristic example, whose structure has been analyzed computationally.
We found a potential period doubling and period halving around an unstable regime, which we assume to be chaotic.
This chaotic regime has been examined further computationally, exposing a spiral structure via special Poincaré section from computational neuroscience called interspike interval bifurcation, which is not visible in the vanilla Poincaré section.
The hereby taken approach can be seen as a basic framework for to guide numerical analyses of solvers for Hodgkin-Huxley type systems. 

All solvers agree on the basic spiral structure, excepting the PIDICN which was unable to unravel higher windings, yielding much noise across the diagram in this area for the given parametrization.
We also observed that the computations of the interspike interval bifurcations with adaptive higher order solvers lead to less noisy looking structures. 
This gives us pointers that these solvers, while not agreeing on solutions due to the potential chaos, still somewhat preserve the character of solutions.

In the future we look forward to analyze stochastic definitions of Hodkin-Huxley type systems.
We also plan to examine geometrical and dynamical properties of the interspike interval bifurcation rigorously, providing a better foundation to understand the properties of numerical solvers for these kind of systems.

\bibliographystyle{plain}

\end{document}